\documentclass[12pt, oneside]{amsart}

\usepackage[utf8]{inputenc}

\usepackage{multirow}
\usepackage{amssymb,amsmath,amsfonts,amsthm}
\usepackage{hyperref}
\usepackage{verbatim}

\newcommand{\prk}{\operatorname{prk}}

\newcommand{\Aut}{\operatorname{Aut}}
\newcommand{\Inn}{\operatorname{Inn}}

\numberwithin{equation}{section}
\newcounter{parag}

\newtheorem{lemma}{Lemma}[section]
\newtheorem{theorem}{Theorem}

\newtheorem{prop}{Proposition}[section]

\theoremstyle{remark}
\newtheorem{rem}{Remark}[section]

\begin{document}

\vspace{1cm}

\title{\textsc{Recognition by element orders for simple linear and unitary groups }}

\author{\textsc{M.A. Grechkoseeva}}
\address{Novosibirsk State University, Pirogova, 1, Novosibirsk 630090, Russia;\newline \hspace*{3mm}  Sobolev Institute of Mathematics, Koptyuga 4, Novosibirsk 630090}
\email{grechkoseeva@gmail.com}

\author{\textsc{A.M. Staroletov}}
\address{Sobolev Institute of Mathematics, Koptyuga 4, Novosibirsk 630090}
\email{staroletov@math.nsc.ru}

\author{\textsc{A.V. Vasil'ev}}
\address{Novosibirsk State University, Pirogova, 1, Novosibirsk 630090, Russia;\newline \hspace*{3mm}  Sobolev Institute of Mathematics, Koptyuga 4, Novosibirsk 630090}
\email{vasand@math.nsc.ru}

\begin{abstract}

For a finite group $G$, let $\omega(G)$ be the set of element orders of $G$ and let $h(G)$ be the number of pairwise nonisomorphic finite groups $H$ with $\omega(H)=\omega(G)$. We say that the recognition problem is solved for $G$ if the number $h(G)$ is known, and if it is finite, then all finite groups $H$ with $\omega(H)=\omega(G)$ are described. We complete the solution of the recognition problem for the finite simple linear and unitary groups. 

\end{abstract}


\begingroup
\let\MakeUppercase\relax 
\maketitle
\endgroup

\section{Introduction}

Like primes in number theory, simple groups lie in the core of finite group theory. Attempts to characterize them in the class of all finite groups in one way or another have been made many times and became especially persistent since the classification theorem was announced and the complete list of simple groups became available.

In the middle of the 1980s Wujie Shi \cite{84Shi, 86Shi} noticed that the smallest nonabelian simple groups $L_2(5)$ and $L_2(7)$\footnote{In the notation of simple classical groups we follow the one-letter notation from the {\em Atlas of Finite Groups}.} can be recognized among all finite groups by the sets of their element orders. Since then and to this day the question of recognizing simple groups by their element orders attracts a lot of attention, see the recent survey \cite{23Survey}.

To deal with the subject and simplify the notation we agree on the following terminology. Given a finite group $G$, we write $\omega(G)$ for the set of element orders of $G$ and refer to this set as the \textit{spectrum} of~$G$. Groups are said to be \textit{isospectral} if their spectra coincide. The number of pairwise nonisomorphic finite groups isospectral to a group $G$ is denoted by~$h(G)$. We say that the {\em recognition problem} is solved for $G$ if the number $h(G)$ is known, and if it is finite, then all the groups isospectral to $G$ are described.

Mazurov and Shi  \cite{12MazShi.t} proved that $h(G)=\infty$ if and only if $G$ is isospectral to a group with nontrivial normal abelian subgroup. In particular, $h(G)$ is finite only if $G$ has trivial solvable radical. Observe that nonabelian simple groups are most basic among the groups having such structure.

Returning to simple groups, it was discovered that the group $L_2(9)$ is not uniquely determined by spectrum, moreover, $h(L_2(9))=\infty$, see the proof in \cite{91BrShi}. This motivated Praeger and Shi  to conjecture in \cite{94PrShi} that $h(L)\in\{1,\infty\}$ for every finite simple group~$L$, see also Problem~12.84 in the {\em Kourovka Notebook} \cite{Kou}.

The conjecture turned out to be correct for the sporadic groups and alternating groups but false for groups of Lie type. The first example of a group $L$ with $h(L)=2$, namely $L=L_3(5)$, was found by Mazurov in \cite{94Maz.t}. He proved that $L$ is isospectral to $\Aut L\simeq L\rtimes\langle \tau\rangle$, where $\tau$ is a graph automorphism of $L$, and there are no other finite groups isospectral to~$L$. Later in \cite{04Zav}, Zavarnitsine showed that for every positive integer $r$ there is a simple $3$\nobreakdash-\hspace{0pt}dimensional linear group $L$ such that $h(L)=r$ and every group $G$ isospectral to $L$ is an almost simple group with socle $L$, that is, up to isomorphism
\begin{equation}\label{eq:almost}
L\simeq\Inn L\leq G\leq\Aut L.
\end{equation}
We call a nonabelian simple group $L$ {\em almost recognizable} if every finite group $G$ isospectral to $L$ satisfies \eqref{eq:almost}. Mazurov conjectured that, with some number of exceptions that can be explicitly described, nonabelian simple groups are almost recognizable.

Mazurov's conjecture was confirmed for sporadic groups: $h(J_2)=\infty$ and $h(L)=1$ otherwise \cite{98MazShi}; for alternating groups: $h(A_6)=h(A_{10})=\infty$ and $h(A_n)=1$ for $n\geq5$ and $n\not\in\{6,10\}$ \cite{13Gor.t}; and for exceptional groups of Lie type: $h({}^3D_4(2))=\infty$ and $L$ is almost recognizable otherwise \cite{14VasSt.t}. It was also confirmed for simple classical groups in the following sense: every classical group of dimension greater than $60$ is almost recognizable \cite{15Vas,15VasGr1}.

Thus, to obtain the complete solution of the recognition problem for simple groups, one needs to address the following two problems. The first one is to describe almost simple groups isospectral to their socles. The second one is to decide which simple classical groups of dimension at most $60$ are almost recognizable. The solution of the first problem was completed by Grechkoseeva in \cite{16Gr.t, 17Gre, 18Gr.t} with substantial use of the description of the spectra of these groups given earlier by Buturlakin in \cite{08But.t, 10But.t}.

The main goal of this article is to solve the second problem for the linear and unitary groups, thus completing for them the solution of the general recognition problem.

\begin{theorem}\label{t:main}
Suppose that $L$ is a finite nonabelian simple linear or unitary group.
If $L$ is one of the groups $L_2(9)$, $L_3(3)$, $U_4(2)$, $U_5(2)$, $U_3(5)$, and $U_3(q)$, where $q$ is a Mersenne prime such that $q^2-q+1$ is also a prime, then $h(L)=\infty$. Otherwise, $h(L)<\infty$, every finite group $G$ isospectral to $L$ is an almost simple group with socle isomorphic to~$L$ and all such $G$ are known.
\end{theorem}

\begin{rem}\label{r:1} The only known Mersenne primes $q$ such that $q^2-q+1$ is also a prime are $3$ and~$7$, and the conjecture is that there are no other such primes. In 2006, Zavarnitsine verified this conjecture for the first 43 Mersenne primes (all known Mersenne primes at that moment), see \cite{06Zav.t}. When we were preparing this article, he kindly checked that the conjecture holds true for all the new 9 Mersenne primes numbers that have been discovered since then.
\end{rem}

\begin{rem} If $L$ is an almost recognizable simple group, then the precise value of $h(L)$ and the description of all groups $G$ isospectral to $L$ can be readily extracted from \cite[Tables~1--9]{23Survey}, in particular, from Table~1 for linear groups and Table~2 for unitary groups. One should proceed as follows: in the table corresponding to the type of $L$, find the section containing the condition of the form ``$n\geq n_0$'' and take $h(L)$ and the description of isospectral groups from this section. We decided not to repeat Tables 1 and 2 of \cite{23Survey} in the present article since they need to be supplemented by explanations of how to use them  from \cite[Section 2.3]{23Survey} and cumbersome notation for outer automorphisms from \cite[Section 2.2]{23Survey}. 
\end{rem}

In fact, Theorem~\ref{t:main} is a direct corollary of the bunch of known results and the following theorem.

\begin{theorem}\label{t:l9l10}
If $L$ is one of the simple groups $L_9(q)$, $L_{10}(q)$, $U_9(q)$, $U_{10}(q)$, where $q$ is odd, then every finite group isospectral to $L$ is an almost simple group with socle isomorphic to~$L$.
\end{theorem}

The formal proof of Theorem~\ref{t:main} will be done in Section~\ref{s:t12}. Here we outline why this general theorem follows from quite particular Theorem~\ref{t:l9l10}. The recognition problem for the groups $L_2(q)$ has been solved in 1994, the last step was done by Brandl and Shi in \cite{94BrShi}. The recognition of $L_3(q)$ and $U_3(q)$ was completed by Zavarnitsine in \cite{04Zav1.t} and \cite{06Zav.t}, respectively. The case of linear and unitary groups in characteristic $2$ was completed by Grechkoseeva and Vasil'ev \cite{08VasGr.t}, and Grechkoseeva and Shi \cite{13GrShi.t}, respectively.

To explain the case of odd characteristic, we need some terminology to be presented. For a finite group $G$, let $\pi(G)$ be the set of prime divisors of the order of $G$ and denote by $t(G)$ the greatest size of a subset $\rho$ of $\pi(G)$ such that $pq\not\in\omega(G)$ for any two distinct primes $p,q\in\rho$. 
Note that the dimension of a simple classical group $L$ and $t(L)$ are linearly equivalent, see Table~\ref{tab:tL} in Section~\ref{s:pre}. 

In 2015, Vasil'ev proved that for every simple classical group $L$ with $t(L)\geq23$, a group $G$ isospectral to $L$ cannot have a composition factor isomorphic to a simple group of Lie type whose defining characteristic distinct from that of~$L$ \cite{15Vas}. Modulo the previous thorough analysis of a possible structure of a finite group isospectral to a simple classical group, see details in \cite[Section~2]{15Vas} or \cite[Section~3]{23Survey}, Vasil'ev's result suffices to prove the almost recognizability of simple linear and unitary groups of dimension at least~$45$ (equivalently, with $t(L)\geq23$) and, with some additional efforts made by Vasil'ev and Grechkoseeva in~\cite{15VasGr1}, to do the same for symplectic and orthogonal groups of dimension greater than~$60$ (this implies $t(L)\geq23$).

Thus, the recognition problem was reduced to the case where $L$ is a simple classical group with $t(L)\leq22$. The analysis of the remaining groups turned out to be quite technically complicated and took more than ten years. The research was naturally divided in two directions. On the one hand, in the series of papers \cite{17Sta,21Sta.t,26Sta.t}, Staroletov was able to significantly lower the upper bound on $t(L)$, solving the recognition problem for the simple linear and unitary groups with $t(L)\geq6$ and symplectic and orthogonal groups with $t(L)\geq14$. On the other hand, Grechkoseeva and her colleagues completed the solution of recognition problem for the simple classical groups with $t(L)\leq4$, see \cite{19GrVasZv,20GrZv.t,25Gr.t,24GrPan.t,26Pan.t,26GreRod_arxiv}. As readily seen from the above, to recognize by spectrum all the simple linear and unitary groups, it suffices to handle the case $t(L)=5$. The groups with this property are exactly the groups $L_9(q)$, $L_{10}(q)$, $U_9(q)$, and $U_{10}(q)$ from Theorem~\ref{t:l9l10}.

It is quite natural to consider all simple classical groups with $t(L)=5$ altogether. Thus, we also prove the following theorem covering the case $t(L)=5$ for symplectic and orthogonal groups.

\begin{theorem}\label{t:main2}
If $L$ is one of the simple groups $S_{10}(q)$, $S_{12}(q)$, $O_{11}(q)$, $O_{13}(q)$, $O_{12}^-(q)$, where $q$ is odd, then every finite group isospectral to $L$ is an almost simple group with socle isomorphic to~$L$.
\end{theorem}

As it follows from the above analysis, to solve the recognition problem for all finite simple groups, it remains to handle the symplectic and orthogonal groups with $6\leq t(L)\leq13$, which makes the complete solution apparently achievable within a range of the one last paper.

\section{Preliminaries}\label{s:pre}

In this section we provide necessary number-theoretical results, as well as results concerning prime graphs and element orders of simple classical groups.

If $a$ and $b$ are integers, then $(a,b)$ denotes the greatest common divisor of $a$, while $[a,b]$ denotes their least common multiple. 
If $r$ is a prime, then $(a)_r$ stands for the $r$-part of $a$,
that is, the highest power of $r$ dividing $a$. The number $a/{(a)_r}$ is called the $r'$-part of $a$ and denoted by $(a)_{r'}$.

Fix an integer $a$ with $|a|>1$. If $r$ is an odd prime and $(a,r)=1$, then $e(r,a)$ denotes the
multiplicative order of $a$ modulo $r$. Define $e(2,a)$ to be $1$ if $4$ divides $a-1$ and to be $2$
if $4$ divides $a+1$. A prime $r$ is said to be a {\it primitive prime divisor} of
$a^i-1$ if $e(r,a)=i$. We write $r_i(a)$ to denote some primitive prime divisor of $a^i-1$, if such
a prime exists, and $R_i(a)$ to denote the set of all such divisors. The existence of primitive divisors for almost all pairs $(a,i)$ was proved by Bang~\cite{86Bang} and Zsigmondy~\cite{Zs}

\begin{lemma}[Bang--Zsigmondy]\label{l:zsigmondy}
Let $a$ be an integer and $|a|>1$. 
If $i$ is a positive integer and
$(a,i)\notin\{(2,1),(2,6),(-2,2),(-2,3),(3,1),(-3,2)\}$, then 
the set $R_i(a)$ is not empty.
\end{lemma}

For $i\geq3$, we denote the product of all primitive prime divisors of $a^i-1$ taken with multiplicities by $k_i(a)$. 
By definition, we have $(k_i(a),k_j(a))=1$ if $i\neq j$.
It follows from \cite{97Roi} that 
\begin{equation}\label{eq:ki}
k_i(a)=\frac{\Phi_i(a)}{(r,\Phi_{(i)_{r'}}(a))},
\end{equation}
where $\Phi_i(x)$ is the $i$th cyclotomic polynomial and $r$ is the largest prime dividing $i$;
moreover, if $(i)_{r'}$ does not divide $r-1$ then $(r,\Phi_{(i)_{r'}}(a))=1$. Recall that $\deg\Phi_n(x)=\varphi(n)$, where $\varphi(x)$ is the Euler totient function.  The properties of cyclotomic polynomials and  \eqref{eq:ki} imply that $k_i(-a)=k_{2i}(a)$ if $i$ is odd and $k_i(-a)=k_{2i}(a)$ if $i$ is a multiple of $4$.

\begin{lemma}\label{l:cyclotomic estimation} Let $i\geq 3$  and $a\in \mathbb Z$ with $b=|a|\geq 2$.
\begin{enumerate}
  \item Let $b\geq k$. If $i$ is a prime, then $\Phi_i(a)<\frac{k}{k-1}\cdot b^{\varphi(i)}$. If $(i)_{2'}$ is a prime power, then $\Phi_i(a)>\frac{k}{k+1}\cdot b^{\varphi(i)}$. 
  \item If $i$ has two odd prime divisors, then $\Phi_i(a)>b^{\varphi(i)}/2$.
\item If $i<105$, then $b^{\varphi(i)}/2<\Phi_i(a)<2b^{\varphi(i)}$.
  
\end{enumerate}
 \end{lemma}

\begin{proof}  (i) If $i=r$ is a prime, then $\Phi_i(x)=(x^r-1)/(x-1)$ and so $$\frac{b^r}{b+1}<\frac{b^r+1}{b+1}\leq\Phi_i(a)\leq \frac{b^r-1}{b-1}<\frac{b^r}{b-1}=\frac{b}{b-1}b^{\varphi(i)}\leq \frac{k}{k-1} b^{\varphi(i)}.$$ If $i=r^t$, then  $\Phi_i(x)=\Phi_r(x^{i/r})$ and therefore  $$\Phi_i(a)=\Phi_r(a^{i/r})>\frac{b^i}{b^{i/r}+1}=\frac{b^{i/r}}{b^{i/r}+1}b^{\varphi(i)}\geq \frac{k}{k+1}b^{\varphi(i)}.$$ 
If $i=2^sr^t$, then  $\Phi_i(x)=\Phi_{r}(-x^{i/2r})$ and we proceed as above.

(ii) If $i=r^ts^u$, then $\Phi_{i}(x)=\Phi_{rs}(x^{i/rs})$, so we may assume that $i=rs$. Then $\Phi_{rs}(x)=\Phi_r(x^s)/\Phi_r(x)$ and 
$$2\Phi_{rs}(x)-x^{(r-1)(s-1)}=\frac{x^{sr+1}-2x^{rs}+x^{sr-r+1}+x^{sr-s+1}-2x-x^{rs-r-s+1}+2}{(x^r-1)(x^s-1)}.$$
It is clear that $2\Phi_{rs}(a)-a^{(r-1)(s-1)}>0$ if $a\geq 2$. If $a<0$, then the numerator in $2\Phi_{rs}(a)-a^{(r-1)(s-1)}$ is equal to $b^{sr+1}+2b^{rs}-b^{sr-r+1}-b^{sr-s+1}+2b-b^{rs-r-s+1}+2$ and so positive too.

(iii) Let $k=\varphi(i)$. Since $i<105$, the coefficients of $\Phi_i(x)$ are in $\{-1,0,1\}$. This implies that $$\Phi_i(a)\leq \frac{b^{k+1}-1}{b-1}<\frac{b^{k+1}}{b-1}\leq 2b^{k}.$$
If $i$ has at most two prime divisors, then we apply (i) or (ii) to get the lower bound. If $i$ has three prime divisors then $(i)_2=2$ and $\Phi_i(a)=\Phi_{i/2}(-a)$. 
\end{proof}

Let $G$ be a finite group. The set $\omega(G)$ is uniquely determined by the subset $\mu(G)$ consisting of element orders that are maximal under divisibility relations. 
The exponent $\exp(G)$ is the least positive integer $k$ such that $g^k=1$ for all $g\in G$. In other words, it is the least common multiple of elements of $\omega(G)$.
Our standard references for the spectra of simple classical groups are \cite{08But.t} and \cite{10But.t} (with corrections from \cite[Lemma 2.3]{16Gr.t}).

Recall that $\pi(G)$ is the set of all prime divisors of the order of $G$. The \emph{prime graph} $GK(G)$ of $G$ is the graph with vertex set $\pi(G)$ in which vertices $r$ and $s$ are adjacent if and only if $r\neq s$ and $rs\in\omega(G)$. A set $\rho\subseteq\pi(G)$ is a \emph{coclique} if the vertices of $\rho$ are pairwise nonadjacent.  We refer to a coclique in $GK(G)$ containing $r\in\pi(G)$ as an \emph{$\{r\}$-coclique}. We denote by $t(G)$ the greatest size of a coclique in $GK(G)$ and by $t(r,G)$ the greatest size of $\{r\}$-cocliques in $GK(G)$. A prime $r\in\pi(G)$ is called \emph{large} with respect to $G$ if $t(r,G)=t(G)$. Our standard references for adjacency in the prime graphs of simple groups are \cite{05VasVd.t} and \cite{11VasVd.t}.

As we mentioned above, we denote simple classical groups according to the \emph{Atlas of Finite Groups}~\cite{85Atlas}. Along with this notation, we write $L_n^+(q)$ for $L_n(q)$ and $L_n^-(q)$ for $U_n(q)$. If $L$ is equal to $L_n^\pm(q)$, $S_{2n}(q)$, $O_{2n+1}(q)$, or $O_{2n}^\pm(q)$, then we say that $L$ is \emph{a group over field of order $q$} (despite the fact that $U_n(q)$ originates from a subgroup of $SL_n(q^2)$) and refer to $n$ as $\prk(L)$. In other words, $\prk(L)$ is the dimension of $L$ if $L$ is linear or unitary and the Lie rank of $L$ otherwise.

It is well known that a prime $r$ lies in $\pi(L_n(q))$ if and only $r$ divides $q$ or $e(r,q)\leq n$. To write similar conditions for other classical groups, the following functions were introduced in \cite{05VasVd.t}:
$$\nu(k)=\begin{cases} 2k& \text{if $k$ is odd},\\
                       k/2& \text{if } (k)_2=2,\\
                       k&\text{if } (k)_2>2,
    
\end{cases} \quad \eta(k)=\begin{cases} k &\text{if $k$ is odd},\\
                       k/2 &\text{if $k$ is even.} 
                       \end{cases}$$

Observe that $R_i(-q)=R_{\nu(i)}(q)$ and $r_i(q)$  divides $q^{\eta(i)}+(-1)^i$.     Now given a simple classical group $L$ over a field of order $q$ and characteristic $p$ and $r\in\pi(L)\setminus\{p\}$, we define $\varphi(r,L)$ as follows:
$$\varphi(r,L)=\begin{cases} e(r,q) &\text{if } L=L_n(q),\\
                             \nu(e(r,q)) &\text{if } L=U_n(q),\\
                             \eta(e(r,q))& \text{otherwise}.    
\end{cases}$$

\begin{lemma}\label{l:adj_criteria}
Let $L$ be a simple classical group a field of order $q$ and characteristic $p$ with $\prk L=n\geq 4$. Let  $r$, $s\in\pi(L)\setminus\{2,p\}$.
Put $k=e(r, q)$, $l=e(s, q)$ and suppose that $\varphi(r,L)\leq\varphi(s,L)$. Then the following statements hold.
\begin{enumerate}
\item  If $L=L^\varepsilon_n(q)$ and $\varphi(r,L)\geq 2$, then $r$ and $s$ are adjacent in $GK(L)$ if and only if $\varphi(r,L)+\varphi(s,L)\leq n$ or $\varphi(r,L)$ divides $\varphi(s,L)$.
\item  If $L\in\{O_{2n+1}(q), S_{2n}(q)\}$, then $r$ and $s$ are adjacent in $GK(L)$ if and only if $\varphi(r,L)+\varphi(s,L)\leq n$ or $\frac{l}{k}$ is an odd integer.
\item If $L=O^\varepsilon_{2n}(q)$, then $r$ and $s$ are adjacent in $GK(L)$ if and only if $2\varphi(r,L)+2\varphi(s,L)\leq 2n-(1-\varepsilon(-1)^{k+l})$, or $\frac{l}{k}$ is an odd integer, or $\varepsilon = +$, $n$ is even, $s\in R_n(q)$, $r\in R_{n/2}(q)$.
\end{enumerate}
\end{lemma}

\begin{proof}
This follows from \cite[Propositions 2.1 and 2.2]{05VasVd.t} and \cite[Propositions 2.4 and 2.5]{11VasVd.t}.
\end{proof}

\begin{lemma}\label{l:adj_criteria1}
Let $L$ be a simple classical group a field of order $q$ and characteristic $p$ with $\prk L=n\geq 4$. Let  $r\in\pi(L)\setminus\{2,p\}$. If $\varphi(r,L)\leq n-2$, then $pr,2r\in\omega(L)$.
\end{lemma}

\begin{proof}
This follows from \cite[Propositions 3.1]{05VasVd.t} and \cite[Section 4]{11VasVd.t}.
\end{proof}

For $\sigma\subseteq\pi(L)\setminus\{p\}$,  
set $E(\sigma,L)=\{e(r,q)~|~r\in\sigma\}$. Let $t(L)\geq 5$ and let $\rho$ be a  coclique of greatest size in $GK(L)$. Since $t(p, L)\leq 4$ by \cite[Table 4]{05VasVd.t}, 
the set $E(\rho, L)$ is well-defined.  We define $J(L)$ to be  the union of $E(\rho, L)$, where $\rho$ runs over all cocliques of greatest size in $GK(L)$, and $E(L)$ to be the intersection of these sets.

\begin{table}[!th]
\caption{Cocliques of greatest size in classical groups}\label{tab:tL}
\begin{center}
{\begin{tabular}{|c|l|c|c|c|}
  \hline
  $L$ & Conditions & $t(L)$ & $E(L)$ & $J(L)\setminus E(L)$ \\
  \hline
 
  $L_n(q)$ & $n\geq9$ is odd and & $\frac{n+1}{2}$ & $\{i\mid \frac{n}{2}<i\leq n\}$ & $\varnothing$\\
           & $(n,q)\neq(9,2),(11,2)$ &  & &\\ 
  & $n=11$ and $q=2$ & 5 & $\{7,8,9,11\}$ & $\{5,10\}$\\
  & $n\geq10$ is even and & $\frac{n}{2}$ & $\{ {i}\mid \frac{n}{2}<i<n\}$ & $\{\frac{n}{2}, n\}$\\
  & $(n,q)\neq(10,2),(12,2)$ &  & &\\ 
   & $n=12$ and $q=2$ & 6 & $\{7,8,9,10,11,12\}$ & $\varnothing$\\    
  \hline
   $U_n(q)$ & $n\geq 9$ is odd & $\frac{n+1}{2}$ & $\{i\mid \frac{n}{2}< \nu\left(i\right)\leq n\}$ & $\varnothing$\\
  & $n\geq 10$ is even & $\frac{n}{2}$ & $\{ {i}\mid \frac{n}{2}< \nu\left(i\right)<n\}$ & $\{\frac{n}{2}, n\}$\\
  \hline
  $S_{2n}(q)$ or  & $n\geq 8$, $n\equiv{0}\pmod 4$& $\frac{3n+4}{4}$ &
  $\{i\mid \frac{n}{2}\leq\eta(i)\leq n\}$ &  $\varnothing$\\
  $O_{2n+1}(q)$ & $n\geq5$, $n\equiv{1}\pmod 4$, and  & $\frac{3n+5}{4}$ &
  $\{i\mid \frac{n}{2}<\eta(i)\leq n\}$ &  $\varnothing$\\
   & $(n,q)\neq(5,2)$  &  &   &  \\   
  &  $n\geq6$, $n\equiv{2}\pmod 4$, and & $\frac{3n+2}{4}$ & $\{i\mid\frac{n}{2}<\eta(i)\leq n\}$ & $\{\frac{n}{2},n\}$\\
    & $(n,q)\neq(6,2)$ & & & \\
  & $n\geq7$, $n\equiv{3}\pmod 4$, and & $\frac{3n+3}{4}$ & $\{i\mid \frac{n+1}{2}<\eta(i)\leq n\}$ &  $\{ \frac{n-1}{2},n-1,$\\
 &  $(n,q)\neq(7,2)$  &  & & $n+1\}$\\
 & $n=6$ and $q=2$ & 5 & $\{3,5,8,10,12\}$ & \\ 
   & $n=7$ and $q=2$ & 6 & $\{5,7,10,12,14\}$ & $\{3,8\}$ \\  
 \hline
 $O_{2n}^+(q)$ & $n\geq 8$, $n\equiv{0}\pmod 4$ & $\frac{3n}{4}$ & $\{i\mid
\frac{n}{2}\leq\eta(i)\leq n,$ & $\varnothing$\\
 & & & $i\neq2n\}$ & \\
 & $n\geq 9$, $n\equiv{1}\pmod 4$ & $\frac{3n+1}{4}$ & $\{i\mid
 \frac{n}{2}<\eta(i)\leq n,$ & $\{n-1, n+1\}$\\
 &  &  & $i\neq2n,n+1\}$ &\\
 & $n\geq10$, $n\equiv{2}\pmod 4$ & $\frac{3n-2}{4}$ & $\{i\mid
\frac{n}{2}<\eta(i)\leq n,$ & $\{\frac{n}{2},n\}$\\
 &  &  & $i\neq2n\}$ &\\
 & $n\geq7$, $n\equiv{3}\pmod 4$ & $\frac{3n+3}{4}$ & $\{i\mid
\frac{n-1}{2}\leq\eta(i)\leq n,$ & $\varnothing$\\
 & & & $i\neq2n,n-1\}$ &\\
 \hline
 $O_{2n}^-(q)$& $n\geq8$, $n\equiv{0}\pmod 4$ & $\frac{3n+4}{4}$ & $\{i\mid
\frac{n}{2}\leq\eta(i)\leq n\}$ & $\varnothing$\\
 & $n\geq 9$, $n\equiv{1}\pmod 4$ & $\frac{3n+1}{4}$ & $\{i\mid
\frac{n}{2}<\eta(i)\leq n,$
 & $\{\frac{n+1}{2}, n-1\}$\\
 & & & $i\neq n,\frac{n+1}{2}\}$ & \\
 & $n=6$, $q=2$ & $5$ & $\{3,8,5,10,12\}$ & $\varnothing$ \\
 & $n=6$, $q>2$ & $5$ & $\{8,5,10,12\}$ &
 $\{3, 6\}$\\
 & $n\geq10$, $n\equiv{2}\pmod 4$ & $\frac{3n+2}{4}$ & $\{i\mid
\frac{n}{2}<\eta(i)\leq n\}$ &
 $\{\frac{n}{2}, n-2, n\}$\\
 & $n\geq 7$, $n\equiv{3}\pmod 4$, and & $\frac{3n+3}{4}$ & $\{i\mid
\frac{n-1}{2}\leq\eta(i)\leq n,$ & $\varnothing$\\
 & $q\neq2$ &  &$i\neq n,\frac{n-1}{2}\}$ &\\
 & $n=7$, $q=2$ & 5 & $\{5,10,12,14\}$ & $\{3,8\}$ \\ 
 \hline
\end{tabular}}
\end{center}
\end{table}

\begin{lemma}\label{l:tL}
Let $L$ be a simple classical group over a field of order $q$ and characteristic $p$ with $t(L)\geq5$. Suppose that $\rho$ is a coclique of greteast size in~$GK(L)$. If $J(L)=E(L)$, then
$E(\rho,L)=E(L)$. If $J(L)\neq E(L)$, then $E(\rho,L)=E(L)\cup\{j\}$ for some $j\in J(L)\setminus
E(L)$. In particular, $|E(L)|\leq t(L)\leq |E(L)|+1$. The sets $E(L)$, $J(L)\setminus E(L)$ and numbers
$t(L)$ are as given in Table~\emph{\ref{tab:tL}}.
\end{lemma}
\begin{proof}
See \cite[Tables 2, 3]{11VasVd.t}. Note that Table 3 of \cite{11VasVd.t} gives a wrong description of $J(L)\setminus E(L)$ for $L=O_{12}^-(q)$: in this case, $J(L)\setminus E(L)=\{3,6\}$ because $r_4(q)$ and $r_{12}(q)$ are adjacent in $GK(L)$ by Lemma~\ref{l:adj_criteria}{\rm(iii)}.  
 
\end{proof}

We will deal mostly with classical groups $L$ satisfying $5\leq t(L)\leq 8$, and these groups are extracted from Table \ref{tab:tL} to Table \ref{tab:5tL8}.

\begin{table}[!hbt]
   \centering
   \caption{The simple classical groups $L$ with $5\leq t(L)\leq 8$} \label{tab:5tL8}
   \begin{tabular}{|c|c|}
\hline
$t(L)$ &  $L$  \\ \hline
5 & $L^{\pm}_{9}(q)$ (except $L_9(2)$), $L^{\pm}_{10}(q)$ (except $L_{10}(2)$), $L^+_{11}(2)$  \\
  & $S_{10}(q)$ ($q\neq 2$), $O_{11}(q)$ ($q\neq 2$), $S_{12}(q)$, $O_{13}(q)$,  $O_{12}^-(q)$, $O_{14}^-(2)$ \\ \hline
6 & $L^{\pm}_{11}(q)$  (except $L_{11}(2)$), $L^{\pm}_{12}(q)$,  $S_{14}(q)$, $O_{15}(q)$, $O_{14}^\pm(q)$ (except $O_{14}^-(2))$, $O_{16}^+(q)$\\ \hline
7 & $L^{\pm}_{13}(q)$, $L^{\pm}_{14}(q)$, $S_{16}(q)$,  $O_{17}(q)$, $O_{16}^-(q)$, $O_{18}^\pm(q)$, $O_{20}^+(q)$ \\ \hline
8 & $L^{\pm}_{15}(q)$, $L^{\pm}_{16}(q)$, $S_{18}(q)$,  $O_{19}(q)$, $S_{20}(q)$,  $O_{21}(q)$, $O_{20}^-(q)$ \\ \hline
\end{tabular}
\end{table}

\begin{lemma}\label{l:tri}
Suppose that $L=L_n^\varepsilon(q)$ and $t(L)\geq 5$. Then the following hold.
\begin{enumerate}
\item If $r\in R_i(\varepsilon{q})$, where $2\leq i<n/2$,
then $t(r,L)\leq i$.
\item If $r\in R_i(\varepsilon{q})$, where $n/3<i<n/2$,
then $t(r,L)=i$.    
\end{enumerate}
\end{lemma}

\begin{proof}
See \cite[Lemma 3.9]{26Sta.t}.
\end{proof}

The two next lemmas are concerned with primes $s_1$ and $s_2$ having
disjoint neighborhoods in $GK(L)$, that is, $s_1$ is not adjacent to $s_2$ and there is no $r\in\pi(L)$ adjacent to both $s_1$ and $s_2$. In Lemma \ref{l:disjoint}, we provide some $s_1$ and $s_2$ with this property for all $L$ with $t(L)\geq 5$, while Lemma \ref{l:unique} is intended to describe all such primes. As we will see in Section \ref{s:structure}, it is convenient to write $S$ and $u$ instead of $L$ and $q$ in Lemma \ref{l:unique}.

\begin{table}[!hbt]
   \centering
   \caption{Vertices of $GK(L)$ with disjoint neighborhoods} \label{tab:disjoint}
$
\begin{array}{|l|c|c|c|c|}
\hline
L&R_1(L)&R_2(L)&m_1(L) & m_2(L)\\
\hline
L_n^\varepsilon(q), n\geq 9&R_n(\varepsilon q)&R_{n-1}(\varepsilon q)& \frac{q^n-\varepsilon^n}{(q-\varepsilon)(n,q-\varepsilon)}&
\frac{q^{n-1}-\varepsilon^{n-1}}{(n,q-\varepsilon)}\\
\hline
S_{2n}(q), O_{2n+1}(q), &\multirow{2}{*}{$R_{2n}(q)$}&\multirow{2}{*}{$R_{n}(q)$}&\multirow{2}{*}{$\frac{q^n+1}{(2,q-1)}$}&\multirow{2}{*}{$
\frac{q^n-1}{(2,q-1)}$}\\
O_{2n+2}^+(q), n\geq 5 \text{ odd}&&&&\\
\hline
S_{2n}(q), n\geq 6 \text{ even}&R_{2n}(q)&R_{n-1}(\epsilon q) & \frac{q^n+1}{(2,q-1)}&\frac{(q^{n-1}-\epsilon)(q+\epsilon)}{(2,q-1)},p(q^{n-1}-\epsilon)\\
\hline
O_{2n+1}(q), n\geq 6 \text{ even}&R_{2n}(q)&R_{n-1}(\epsilon q) & \frac{q^n+1}{(2,q-1)}&\frac{(q^{n-1}-\epsilon)(q+\epsilon)}{(2,q-1)},\frac{p(q^{n-1}-\epsilon)}{(2,q-1)}\\
\hline 
O_{2n}^\varepsilon(q),   n\geq 5 \text{ odd} & R_{2n-2}(q)& R_n(\varepsilon q)&\frac{(q^{n-1}+1)(q+\varepsilon)}{(4,q-\varepsilon)}& \frac{q^n-\varepsilon}{(4,q-\varepsilon)}\\
\hline
O_{2n}^-(q),  n\geq 6 \text{ even } &R_{2n}(q)& R_{n-1}(\epsilon q) &\frac{q^n+1}{(2,q-1)} & \frac{(q^{n-1}-\epsilon)(q+\epsilon)}{(2,q-1)}\\
\hline
\end{array}
$
\end{table}

\begin{lemma}\label{l:disjoint}
Let $L$, $R_1(L)$ and $R_2(L)$ be as in Table~\emph{\ref{tab:disjoint}} (if $n$ is even and $L$ is one of $S_{2n}(q)$, $O_{2n+1}(q)$, or $O_{2n}^-(q)$, then $\epsilon$ is an arbitrary element of $\{+1,-1\}$) and let $p$ be the defining characteristic of $L$. Then $R_i(L)\neq \varnothing$. If $i\in\{1,2\}$ and $s_i\in R_i(L)$, then the following hold.

\begin{enumerate}
\item The neighborhoods of $s_1$ and $s_2$ in $GK(L)$ are disjoint.
\item The elements of $\mu(L)$ divisible by $s_i$ are precisely the numbers specified in the column titled by ``$m_i(L)$'' in Table~\emph{\ref{tab:disjoint}}. In particular, $s_i$ is adjacent only to prime divisors of these numbers in $GK(L)$.
\end{enumerate}
\end{lemma}

\begin{proof} Observe that $R_i(L)$ is of the form $R_j(\pm q)$, where either $j\geq 7$ or $j=5$, so $R_i(L)\neq \varnothing$ by Lemma \ref{l:zsigmondy}. 

By \cite[Section 3]{05VasVd.t}, it follows that $s_i$ is not adjacent to $p$ in $GK(L)$ unless $L=S_{2n}(q)$ or $O_{2n+1}(q)$ with $n$ even and $i=2$. Using the description of the spectra of classical groups, it is not hard to determine all $b\in\mu(L)$ that are coprime to $p$ and divisible by $s_i$. A useful observation here is that $[q^{n-1}-\epsilon,q+\epsilon]=(q^{n-1}-\epsilon)(q+\epsilon)/(2,q-1)$ if $n$ is even and $\epsilon\in\{+1,-1\}$. 

Suppose that $L=S_{2n}(q)$ or $O_{2n+1}(q)$ with $n$ even, $i=2$ and $ps_i$ divides some $b\in\omega(L)$. If $L=S_{2n}(q)$ or $p=2$, then $b$ divides $p(q^{n-1}-\epsilon)$ and $p(q^{n-1}-\epsilon)\in\omega(L)$ by \cite[Corollary 2]{10But.t} or \cite[Corollary 3]{10But.t}, respectively.  If  $L=O_{2n+1}(q)$ and $p$ is odd, then $b$ divides  $p(q^{n-1}-\epsilon)/2$ and $p(q^{n-1}-\epsilon)/2\in\omega(L)$ by \cite[Corollary 6]{10But.t}. So (ii) holds. Now (i) follows from (ii) since every number in $m_1(L)$ is coprime to every number in $m_2(L)$. Also (i) can be verified directly by application of the adjacency criterion. 
\end{proof}

\begin{lemma}\label{l:unique}
Let $S$ be a simple classical group over a field of order $u$ and $t(S)\geq 5$. Assume that $u>3$ in all items except {\rm(v)}. Suppose that $s_1\in\pi(S)$ and $s_2\in \pi(S)$ have disjoint neighborhoods in $GK(S)$. Then, up to renumbering of $s_1$ and $s_2$, the following hold.

\begin{enumerate}
\item If $S=L_m^\tau(u)$, where $m$ is even, then  $s_1\in R_{m-1}(\tau u)$. If, in addition, $u-\tau$ does not divide $m$, then $s_2\in R_m(\tau u)$.

\item If $S=L_m^\tau(u)$, where $m$ is odd, then $s_1\in R_m(\tau u)$. If, in addition, $3$ divides $m$,  then $s_2\in R_{m-1}(\tau u)\cup R_{m-2}(\tau u)$.

\item If $S$ is one of $S_{2m}(u)$, $O_{2m+1}(u)$, or $O_{2m+2}^+(u)$,  where $m$ is odd, then $s_1\in R_{2m}(u)$ and $s_2\in R_{m}(u)$.

\item If $S=O_{2m}^\tau(u)$, where $m$ is odd, then  $s_1\in R_m(\tau u)$. In, in addition, $\tau u\neq +5$,  then $s_2\in R_{2(m-1)}(\tau u)$.

\item If $S$ is one of 
$S_{2m}(u)$, $O_{2m+1}(u)$,  or $O_{2m}^-(u)$, where 
$m$ is even, then $s_1\in R_{2m}(u)$. If, in addition, $(m)_2=2$, then $s_2\in R_{m-1}(\tau u)\cup R_{2(m-1)}(u)$ if $S\neq O_{2m}^-(u)$ and $s_2\in R_{m-1}(\tau u)\cup R_{2(m-1)}(u)\cup R_{2(m-2)}(u)$ otherwise.
\end{enumerate}
\end{lemma}

\begin{proof}
Observe that $R_1(u)\neq \varnothing$ and $R_2(u)\neq \varnothing$ in all items except (v) by Lemma \ref{l:zsigmondy} and the assumption $u>3$.

(i), (ii) Let $S=L_m^\tau(u)$. Then $m\geq 9$.  Set $k=m-1$ if $m$ is even and $k=m$ if $m$ is odd. Every $r_2(\tau u)$ is adjacent to every prime in $\pi(S)\setminus R_{k}(\tau u)$, so we may assume that $s_1\in R_k(\tau u)$. 

If $m$ is even and  $u-\tau$ does not divide $m$, then there is $r\in \pi(u-\tau)$ such that $(u-\tau)_r>(m)_r$. Then by \cite[Propositions 4.1 and 4.2]{05VasVd.t}, it follows that $r$ is adjacent to every prime in $\pi(S)\setminus R_m(\tau u)$. Hence $s_2\in R_m(\tau u)$.

If $m$ is odd and divisible by $3$, then $s_1$ is adjacent to $r_3(\tau u)$, which in turn adjacent to every $s$ not lying in $R_{m-1}(\tau u)\cup R_{m-2}(\tau u)$.

(iii) Here $m\geq 5$.  By \cite[Table 6]{05VasVd.t} and \cite[Propositions 2.4, 2.5]{11VasVd.t}, it follows that $r_1(u)$ is adjacent to every $s\in \pi(S)\setminus R_{2m}(u)$, while $r_2(u)$ is adjacent to every $s\in \pi(S)\setminus R_{m}(u)$. 

(iv) Here $m\geq 7$. Arguing as in (iii), we see that $r_2(u)$ is adjacent to every $s\in \pi(S)\setminus R_{m}(\tau u)$, so $s_1\in R_m(\tau u)$. Suppose that either $R_1(\tau u)\neq \{2\}$,  or $R_1(\tau u)=\{2\}$ and $u-\tau\not\equiv 4\pmod 8$. Then there is $r\in R_1(\tau u)$ adjacent to every $s\in \pi(S)\setminus R_{2(m-1)}(u)$, which yields $s_2\in R_{2(m-1}(u)$. If $R_1(\tau u)=\{2\}$ and $u-\tau\equiv 4\pmod 8$, then $u-\tau=4$ and hence $\tau u=+5$.

(v)  Here $m\geq 5$. If $r\in\pi(u^2-1)$, then $r$ is adjacent to every $s\in\pi(S)\setminus R_{2m}(u)$, so $s_1\in R_{2m}(u)$. Suppose that $(m)_2=2$. Then $s_1$ is adjacent to $r_4(u)$. If $s_2$ is not as specified then $s_2$ and $r_4(u)$ are adjacent.
\end{proof}

\begin{lemma}
\label{l:orders}
Suppose that $L$ is a simple classical group over a field of order $q$. Then element orders of $L$ do not exceed $\frac{q}{q-1}q^l$, where $l$ is the Lie rank of $L$. 
\end{lemma}

\begin{proof}
See \cite[Lemma~1.3]{09VasGrMaz.t}.
\end{proof}

\begin{table}[hbt]
   \caption{Bounds on the exponents of some classical groups} \label{tab:exp}
   \begin{tabular}{|c|c|c|c|c|c|}
\hline 
Type of $L$ & $(\alpha(L),\beta(L),\gamma(L))$ & Type of $L$ & $(\alpha(L),\beta(L),\gamma(L))$ \\ \hline
$L^\pm_{9}$ &  $(32, 86, 28)$ & $S_{10}$, $O_{11}$ & $(5, 20, 20)$  \\ 
$L^\pm_{10}$ &  $(6, 64, 32)$ & $S_{12}, O_{13}$, $O^-_{12}$ & $(4, 16, 24)$   \\ 
$L^\pm_{11}$ &  $(118, 143, 42)$ & $S_{14}, O_{15}$, $O^+_{16}$ & $(5, 29, 36)$ \\ 
$L^\pm_{12}$  & $(15, 157, 46)$ &  $S_{16}, O_{17}$, $O^-_{16}$ & $(10, 29, 44)$ \\ 
$L^\pm_{13}$ & $(370, 342, 58)$ &$O^{\pm}_{14}$ & $(7, 37, 30)$\\ 
$L^\pm_{14}$ & $(15, 247, 64)$ &  $O^{\pm}_{18}$ & $(10, 65, 50)$   \\ 
 & &    $O^+_{20}$ & $(8, 49, 56)$  \\
\hline

\end{tabular}
\end{table}

\begin{lemma}\label{l:exp}
Suppose that $L$ is a simple classical group
over a field of order $q$ and characteristic $p$ 
such that $5\leq t(L)\leq 7$.
Then $\frac{1}{\alpha(L)}\cdot p\cdot q^{\gamma(L)}\leq\exp(L)\leq\beta(L)\cdot p\cdot q^{\gamma(L)}$,
where numbers $\alpha(L), \beta(L)$, and $\gamma(L)$ are listed in Table~{\rm \ref{tab:exp}}.
\end{lemma}

\begin{proof}
See \cite[Proposition 2.15]{26Sta.t}.
\end{proof}

\section{The structure of groups isospectral to classical groups}\label{s:structure}

In this section, we collect results about the structure of finite groups isospectral to classical groups in odd characteristic.

\begin{lemma}\label{l:reduction}
Suppose that $L$ is a simple classical group over a field of odd characteristic $p$, $t(L)\geq 5$ and $\omega(G)=\omega(L)$. If $G$ is not an almost simple group with socle isomorphic to $L$, then $S\leq \overline G=G/K\leq \Aut S$, where $K$ is the solvable radical of $G$ and $S$ is a simple classical group in characteristic $v\neq p$. Furthermore, $K$ is nilpotent and the following hold:
\begin{enumerate}
     \item if $\rho$ is a coclique in $GK(G)$ with $|\rho|\geq3$, then at most one prime in $\rho$ divides $|K|\cdot|\overline{G}/S|$;
    \item if $r\in\pi(G)$ is not adjacent to $2$ in $GK(G)$, then $r$ does not divide  $|K|\cdot|\overline{G}/S|$.
\end{enumerate}
\end{lemma}

\begin{proof}
Since $t(2,L)\geq 2$ by \cite[Theorem 7.1]{05VasVd.t}, the assertion of the lemma follows from \cite{05Vas.t}, except  nilpotency of $K$ and restrictions on the type of $S$. The group $K$ is nilpotent by \cite{20YanGrVas}. By \cite[Theorems 3.8 and 3.6]{23Survey}, the group $S$ is classical and if the defining characteristic of $S$ is equal to $p$, then $S\simeq L$ and $K=1$. 
\end{proof}

\begin{lemma}\label{l:rest on t(S)}
Let $L$ and $S$ be as in Lemma {\rm \ref{l:reduction}}. 
\begin{enumerate}
\item If $L$ is linear or unitary with $t(L) = 5$ or $L$ is  symplectic or orthogonal with $t(L)\geq 7$, then $0\leq t(S)-t(L)\leq 1$.
\item If $L$ is symplectic or orthogonal with $5\leq t(L)\leq 6$, then $0\leq t(S)-t(L)\leq 2$.
\end{enumerate}
\end{lemma}

\begin{proof}
See \cite[Theorem 3]{26Sta.t}.
\end{proof}

In the rest of this section, we assume that $L$ is a finite simple classical group over a field of order $q$ and characteristic $p\neq 2$ with $5\leq t(L)\leq 13$,  $G$ is a finite group such that $\omega(G)=\omega(L)$ but $G$ is not an almost simple group with socle $L$. By Lemma \ref{l:reduction}, we have 
\begin{equation}
S\leq \overline G=G/K\leq \Aut S,
\end{equation} where $K$ is the solvable radical of $G$ and  $S$ is a simple classical group over a field of order $u$ and characteristic $v\neq p$. Also in this section $n=\prk(L)$ and $m=\prk(S)$, if not stated otherwise.

\begin{lemma}\label{l:r in K}
If $r\in\pi(K)\setminus\{v\}$, then $t(r,L)=2$ and $r$ divides $p(q^2-1)$. If $v\in\pi(K)$, then $t(v,L)\leq 3$.
\end{lemma}

\begin{proof}
See \cite[Lemma 4.1]{26Sta.t} and \cite[Lemma 1.7]{26GreRod_arxiv}.
\end{proof}

\begin{lemma}\label{l:G/S-linear}
Let $L=L^\varepsilon_n(q)$, where $n\in\{9,10\}$ and $\varepsilon\in\{+,-\}$. 
If $r\in R_i(\varepsilon{q})\cap \pi(\overline{G}/S)$ and $4\leq i\leq n$, then  one of the following holds:
\begin{enumerate}
\item $i=4$ and $r=5$;
\item $i=6$, $r=31$, $n=9$, and $(q-\varepsilon,5)=5$.
\end{enumerate}
\end{lemma}

\begin{proof}
See \cite[Proposition 4.3]{26Sta.t}.
\end{proof}

\begin{lemma}\label{l:t(S)>=t(L)}
Let $r\in\pi(L)\setminus\{p\}$.
\begin{enumerate}
\item Suppose that $\varphi(r,L)\geq 4$ if $L$ is a linear or unitary group and $\varphi(r,L)\geq 3$ otherwise. If $r\in\pi(S)$, then $t(r, S)\geq t(r, L)$. 
\item Suppose that $r$ is large with respect to $L$ and $r\not\in R_6(\varepsilon q)$ if $L=L_n^\varepsilon(q)$. Then $r$ is coprime to $v\cdot|K|\cdot |\overline G/S|$, $r\in\pi(S)$ and  
$t(r,S)\geq t(L)$.
\end{enumerate}
\end{lemma}

\begin{proof}
Item (i) is proved in \cite[Lemma 4.5(iii)]{26Sta.t}. Suppose that $r$ is large with respect to $L$. According to Table \ref{tab:tL}, we see that $\varphi(r,L)\geq 4$ if $L$ is linear or unitary and $\varphi(r,L)\geq 3$ if $L$ is symplectic or orthogonal. By \cite[Lemma 4.5(i)]{26Sta.t}, it follows that $r$ is coprime to $|\overline G/S|$. Also $r\not\in\pi(K)$ by Lemma \ref{l:r in K}. Thus $r\in\pi(S)$ and applying (i) yields $t(r,S)\geq t(L)$. In particular, $r\neq v$ because $t(v,S)\leq 4$.
\end{proof}

\begin{lemma}\label{l:ts=tl}
Let $t(S)=t(L)$. Suppose that $i\in J(L)$ and $k_i(q)\neq k_6(\varepsilon q)$ if $L=L_n^\varepsilon(q)$. Then either $k_i(q)$ divides $k_j(u)$ for some $j\in E(S)$
or $k_i(q)$ divides $k_{j_1}(u)k_{j_2}(u)$,
where $j_1,j_2\in J(S)\setminus E(S)$.
\end{lemma}

\begin{proof}
By Lemma \ref{l:t(S)>=t(L)}(ii), it follows that $r_i(q)\in\pi(S)$ and $t(r_i(q),S)=t(S)$. Since all elements of $R_i(q)$ are large with respect to $S$ and pairwise adjacent in $GK(S)$, either 
there is $j\in E(S)$ such that $R_i(q)\subseteq R_j(u)$, or 
$R_i(q)\subseteq \cup_{j\in J(S)\setminus E(S)}R_j(u)$. The rest of the proof repeats the corresponding part of the proof of \cite[Lemma 5.1]{26Sta.t} since the latter does not uses the assumption that $L$ is linear or unitary and $t(L)\geq 6$.
\end{proof}

\begin{lemma}\label{l:ts=tl+1}
Suppose that $L$ is symplectic or orthogonal and  $S=L_m^\tau(u)$, where $m$ is odd.
Then one of the following holds:
\begin{enumerate}
\item for every $i\in E(L)$,  every $r_i(q)$ is large with respect to $S$ and $k_i(q)$ divides $k_j(\tau u)$ for some $j$ with $(m+1)/2\leq j\leq m$;
\item there is $k\in E(L)$ and $r\in R_k(q)$ such $\varphi(r,S)\leq (m-1)/2$ and 
for every $i\in J(L)\setminus\{k\}$, $k_i(q)$ divides $k_j(\tau u)$ for some $j$ with $(m+1)/2\leq j\leq m$. 
\end{enumerate}

If, in addition, $t(S)=t(L)+1$, then in {\rm(ii)}, we have $r\in R_{(m-1)/2}(\tau u)$ 
and $j\neq (m-1)/2$, $m-1$.
\end{lemma}

\begin{proof}
Since $t(S)\geq t(L)+1\geq 6$, it follows that $S\neq L_{11}(2)$ and so $J(S)=E(S)$ according to Table~\ref{tab:tL}. 

Let $i\in E(L)$. If every element of $R_i(q)$ is large with respect to $S$, then arguing as in the proof of Lemma \ref{l:ts=tl}, we conclude that $k_i(q)$ divides $k_j(u)$ for some $j$ with $(m+1)/2\leq j\leq m$. 

Suppose that for some $k\in E(L)$, there is $r\in R_{k}(q)$ such that $r$ is not large with respect to $S$. Then $\varphi(r,L)\leq (m-1)/2$ by Table~\ref{tab:tL}. If $i\in E(L)\setminus\{k\}$, then $r_i(q)$ and $r$ are not adjacent in $GK(L)$ and so in $GK(S)$.
By Lemma \ref{l:adj_criteria}, this implies that $r_i(q)\in R_j(\tau u)$ for some $j$ with $(m+1)/2\leq j\leq m$ and then $k_i(q)$ divides $k_j(u)$.  

If $t(S)=t(L)+1$, then $t(r,S)\geq t(L)=t(S)-1$. Therefore, $t(r,S)=t(S)-1=(m-1)/2$ and so $r\in R_{(m-1)/2}(\tau u)$ by Lemma \ref{l:tri}. Then $r$ is adjacent to prime divisors of $u^{m-1}-1$ and therefore $j\neq (m-1)/2,m-1$.
\end{proof}

\begin{lemma}\label{l:s2n_even} Let $L$ be of the groups $S_{2n}(q)$, $O_{2n+1}(q)$, or $O_{2n}^-(q)$, where $n\geq 6$ is even. 
\begin{enumerate}
    \item If $S$ is one of $S_{2m}(u)$, $O_{2m+1}(u)$, $O_{2m+2}^+(u)$, $O_{2m}^\tau (u)$,  where $m$ is odd and $u>3$, then $S=O_{2m}^+(5)$.
    \item  If $S=L_m^\tau(u)$, where $m$ is even and $u>3$, then $u-\tau$ divides $m$ and $k_{2n}(q)$ divides $k_{m-1}(\tau u)$.
    \item  If $S=L_m^\tau(u)$, where $m$ is odd and $u>3$, then $k_{2n}(q)$ divides $k_{m}(\tau u)$.
    \item  If $S$ is one of $S_{2m}(u)$, $O_{2m+1}(u)$, or $O_{2m}^-(u)$, where $m$ is even, then $k_{2n}(q)$ divides $k_{2m}(u)$.
\end{enumerate}
\end{lemma}

\begin{proof}
For every $\epsilon \in\{+,-\}$, the primes $r_{2n}(q)$ and $r_{n-1}(\epsilon q)$ divide $|S|$ by Lemma \ref{l:t(S)>=t(L)} and have no common neighbors in $GK(S)$ by Lemma \ref{l:disjoint}.
By Lemma \ref{l:unique}, in (i) either both $r_{n-1}(q)$ and $r_{2n-2}(q)$ lie in the same $R_j(u)$, which is impossible, or $S$ is as stated. Similarly, in (ii)--(iv),
we have that $k_{2n}(q)$ divides $k_{m-1}(\tau u)$, or $k_{m}(\tau u)$, or $k_{2m}(u)$, respectively.  Moreover, if  $u-\tau$ does not divide $m$ in (ii), then we get a contradiction as in (i). 
\end{proof}

\begin{lemma}\label{l:s2n_odd} Let $L$ be of the groups $S_{2n}(q)$, $O_{2n+1}(q)$, or $O_{2n+2}^+(q)$, where $n\geq 5$ is odd. Suppose that $S=L_m^\tau(u)$, where $u>3$, and there is $\epsilon \in\{+,-\}$ such that $R_n(\epsilon q)\not\subseteq R_m(\tau u)\cup R_{m-1}(\tau u)$. Then $q\equiv \epsilon\pmod 4$ and the following hold.
\begin{enumerate}
    \item  If $m$ is even, then $u-\tau$ divides $m$ and $k_{n}(-\epsilon q)$ divides $k_{m-1}(\tau u)$.
    \item  If $m$ is odd,  then $k_{n}(-\epsilon q)$ divides $k_{m}(\tau u)$.
\end{enumerate}
\end{lemma}

\begin{proof}
The primes $r_{2n}(q)$ and $r_{n}(q)$ divide $|S|$ by Lemma \ref{l:t(S)>=t(L)} and have no common neighbors in $GK(S)$ by Lemma \ref{l:disjoint}. By Lemma \ref{l:unique}, 
it follows that $k_{n}(-\epsilon q)$ divides $k_{l}(\tau u)$, where $l\in\{m,m-1\}$ and $l$ is odd. Moreover, $u-\tau$ divides $m$ in (i). 

If $q\equiv -\epsilon\pmod 4$, then $2r_n(\epsilon q)\not\in\omega(L)$ and so $R_n(\epsilon q)\subseteq R_m(\tau u)\cup R_{m-1}(\tau u)$ by Lemma \ref{l:reduction}(ii) and \cite[Tables 4 and 6]{05VasVd.t}. 
\end{proof}

\begin{lemma}\label{l:d-estim}
Suppose that $d=t(S)-t(L)>0$.  If $t(S)\geq7$, then $u^{3d}<q^2$.
\end{lemma}

\begin{proof}
See \cite[Lemma 4.11]{26Sta.t}.
\end{proof}

\section{Proof of Theorems \ref{t:main} and \ref{t:l9l10}}\label{s:t12}

We begin with the proof of Theorem \ref{t:l9l10}. Suppose that $L=L_n^\varepsilon(q)$, where $n\in\{9,10\}$, $\varepsilon\in\{+,-\}$, and $q$ is a power of an odd prime $p$. Suppose that $\omega(G)=\omega(L)$ and $G$ is not an almost simple group with socle isomorphic to $L$. By Lemma \ref{l:reduction}, we have $$S\leq \overline G=G/K\leq \Aut S,$$ where $S$ is a simple classical group over a field of order $u$ and characteristic $v\neq p$. By Lemma \ref{l:rest on t(S)}, we have $t(S)=6$ or $t(S)=5$. In particular,  Table \ref{tab:5tL8} shows that the Lie rank of $S$ is at most $11$.

Let $i$ be an integer such that $5\leq i\leq n$ and $i\neq 6$. Then $r_i(\varepsilon q)$ is large with respect to $L$ by Table \ref{tab:tL}. By Lemma \ref{l:t(S)>=t(L)}, it follows that $R_i(\varepsilon q)\cap(\{v\}\cup \pi(K)\cup\pi(\overline G/S))=\varnothing$. 
In particular, $k_i(\varepsilon q)\in\omega(S)$. Furthermore, if $n=10$, then $k_5(q)k_{10}(q)\in \omega(S)$ since $k_5(q)k_{10}(q)\in\omega(L)$. Also $t(r_i(\varepsilon q),S)\geq 5$.

\begin{lemma}\label{l:u>3}
We may assume that  $q\geq 5$ and $u\geq 4$. Moreover, if $q=5$, then $u\geq 7$.
\end{lemma}

\begin{proof}  
Suppose that $q=3$. Then $\exp(L)<64\cdot 3^{33}$ by Lemma \ref{l:exp}. By the same lemma and Table \ref{tab:5tL8} , we have $\exp(S)>\frac{2}{5}u^{20}$, and so $u\leq 7$. Now it is easy 
verify that there is  $r_5\in R_5(u)\cap \{31,71,2801\}$ and $r_7\in R_7(-u)\cap \{43,29,113\}$. Then $e(r_5,3)\geq 30$ and $e(r_7,3)\geq 28$ and, therefore, $\pi(S)\not\subseteq\pi(L)$, a contradiction.

To prove that $u\geq 4$, we use the fact that $k_9(\varepsilon q)\in\omega(S)$. Since $k_9(-5)=5167$, $k_9(5)=19\cdot 829$ and $e(r,u)\geq 207$ for $r\in\{5167, 829\}$ and $u\in\{2,3\}$, it follows that $q=5$ yields $u>4$. Similarly, if $q\in\{7,9\}$, then 
$$k_9(\varepsilon{q})\in\{
37\cdot 1063, 117307, 19\cdot37\cdot757,
 530713\}$$
and we find $r\in R_9(\varepsilon q)$ such that $e(r,u)\geq 531$ for $u\in\{2,3\}$
(note that $u\neq3$ if $q=9$ since $v\neq p$).  
If $q\geq 11$, then $$k_9(\varepsilon q)=\frac{\Phi_9(\varepsilon{q})}{(3,q-\varepsilon)}\geq\frac{11}{12}\cdot\frac{11^6}{3}=\frac{11^7}{36}$$ by \eqref{eq:ki} and Lemma \ref{l:cyclotomic estimation}(i). Applying Lemma \ref{l:orders}, we  conclude that $11^7/36<u^{12}/(u-1)$, whence $u>3$. 
\end{proof}

In what follows, we will repeatedly use the fact that  $r_n(\varepsilon q)$ and $r_{n-1}(\varepsilon q)$ lie in $\pi(S)$  and have disjoint neighborhoods in $GK(S)$ by Lemma \ref{l:disjoint}. Also we will bound $k_i(q)$ and $k_j(u)$ using Lemma \ref{l:cyclotomic estimation} and Lemma \ref{l:u>3}, in particular, we will use that $$k_9(\varepsilon{q})=\frac{\Phi_9(\varepsilon{q})}{(3,q-\varepsilon)}\geq\frac{5q^6}{6(3,q-\varepsilon)}\quad\text{and}\quad k_7(u)=\frac{\Phi_7(u)}{(7,u-1)}\leq\frac{4u^6}{3}.$$

\begin{lemma}  $S$ is neither symplectic nor orthogonal. 
\end{lemma}

\begin{proof}

Suppose that $S\in\{O_{15}(u), S_{14}(u), O^+_{16}(u)\}$. By Lemma \ref{l:unique}, it follows that one of  $k_n(\varepsilon q)$ and $k_{n-1}(\varepsilon q)$ divides $k_7(u)$, while the other divides $k_{14}(u)$.

Let $n=9$. Then $k_9(\varepsilon q)$ divides $k_7(\tau u)$ for some $\tau\in\{+,-\}$. Hence $5q^6/18<4u^6/3$, or, equivalently, \begin{equation}\label{e:div}
q^6<(24/5)\cdot u^6.    
\end{equation} By Lemma \ref{l:exp}, we have  
\begin{equation}\label{e:ex} \frac{2}{5}u^{36}\leq\exp S\leq\exp L\leq 86q^{29}.\end{equation} 
Thus $q^{36}<(24/5)^6u^{36}<(24/5)^6\cdot 5\cdot 43q^{29}$, whence $q\leq 7$. Then \eqref{e:ex} yields $u\leq 5$ and, in turn, \eqref{e:div} yields $q=5$. This contradicts Lemma \ref{l:u>3}. 

Let $n=10$. Then $k_{10}(\varepsilon q)$ divides $k_{7}(\tau u)$ for some $\tau\in\{+,-\}$. By Lemma \ref{l:disjoint}(ii), only divisors of $(u^7-\tau)/(2,u+\tau)$ are adjacent to $r_{7}(\tau u)$ in $GK(S)$. Since $t(r_{5}(\varepsilon q),S)\geq 5$ and $t(r,S)<5$ for all $r\in \pi(u-\tau)$, it follows that $k_{10}(q)k_5(q)$ divides $k_{7}(\tau u)$, which yields $q^8/5<4u^6/3$. On the other hand, $2u^{36}<5\cdot 64q^{33}$ by Lemma \ref{l:exp}. Thus $q^{48}<(20/3)^6u^{36}<
(20/3)^6\cdot 5\cdot 32q^{33}$, whence $q<3$.

Suppose that $S=O^\tau_{14}(u)$ and assume for a while that $\tau u\neq +5$. Applying Lemma \ref{l:unique}, we see that one of $k_n(\varepsilon q)$ and $k_{n-1}(\varepsilon q)$ divides $k_7(\tau u)$, while the other divides $k_{12}(u)$. Suppose that $k_9(\varepsilon q)$ divides $k_{12}(u)$. Then 
$5q^6/18<u^4-u^2+1<u^4$ and hence $q^6<(18/5)u^4$. By Lemma \ref{l:exp}, we have $2u^{30}<7\cdot 64q^{33}$. Thus $q^{90}<(18/5)^{15}u^{60}<(18/5)^{15}(7\cdot 32)^2q^{66}$, whence $q=3$, a contradiction.

Let $n=10$. Then $k_{10}(\varepsilon q)$ divides $k_{12}(u)$. 
By Lemma \ref{l:disjoint}(ii), only prime divisors of $(u^6+1)(u+\tau)$ are adjacent to $r_{12}(u)$ in $GK(S)$. If $r\in \pi((u^2+1)(u+\tau))$ and $rs\not\in\omega(S)$, then $s\neq 2,v$ by Lemma \ref{l:adj_criteria1} and $e(s,\tau u)\in\{7, 5\}$ by Lemma \ref{l:adj_criteria}, so $t(r,S)\leq 3$. It follows that  $k_{10}(q)k_5(q)$ divides $k_{12}(u)$. So $q^8/5<u^4$ and, as above, this contradicts $2u^{30}<7\cdot 64q^{33}$.

Let $n=9$. Then $k_8(q)$ divides $k_{12}(u)$ and therefore $(u^6+1)(u+\tau)/(4,u-\tau)$ divides $(q^8-1)/(9,q-\varepsilon)$ by Lemma \ref{l:disjoint}(ii). Suppose that $k_8(q)=k_{12}(u)$. Then $(q^4-1)/2=u^2(u^2-1)$ and $(u^2+1)(u+\tau)/(4,u-\tau)$ divides $2(q^4-1)/(9,q-\varepsilon)$. Since $k_4(u)=(u^2+1)/(2,u-\tau)$ is odd, it follows that $k_4(u)$ divides $u^2(u^2-1)$, a contradiction. Thus $k_8(q)$ is a proper divisor of $k_{12}(u)$. Every prime divisor of $k_{12}(u)$ is equal to $1$ modulo $12$, so $k_8(q)\leq k_{12}(u)/13$, which yields $6q^4<u^4$. This contradicts the inequality $2u^{30}<7\cdot 86q^{29}$.

If $S=O^+_{14}(5)$, then one of the numbers $k_n(\varepsilon q)$ and $k_{n-1}(\varepsilon q)$ is equal to the prime $k_7(5)=19531$. Since $k_7(5)\not\equiv 1\pmod 8$, in fact,
$k_9(\varepsilon q)=k_7(5)$ or $k_{10}(\varepsilon q)=k_7(5)$. It is easy to check that this is impossible. 

Suppose that $S\in\{S_{10}(u), O_{11}(u)\}$. Then by Lemma \ref{l:unique}, one of $k_n(\varepsilon q)$ and $k_{n-1}(\varepsilon q)$ divides $k_5(u)$, while the other divides $k_{10}(u)$. 

If $n=10$, then $k_{10}(q)$ divides $k_5(\tau u)$ for some $\tau\in\{+,-\}$. Since $r_5(q)$ is large with respect to $S$ by Lemma \ref{l:ts=tl} and $J(S)=E(S)$, it follows that 
$k_{10}(q)k_5(q)$ divides $k_5(\tau u)$. Hence $q^8/d<4u^4/3$, where $d=(5,q^2-1)$. On the other hand, $u^{20}<5\cdot 32q^{33}$. Thus $q^{40}<(4d/3)^5u^{20}<(4d/3)^5\cdot 5\cdot 32q^{33}$, whence $q\leq 7$. Then $d=1$ and $q<3$. 

If $n=9$, then one of $k_7(\varepsilon q)$ and $k_5(\varepsilon q)$
divides $k_j(u)$ for some $j\in\{3,6\}$ by Lemma \ref{l:ts=tl}. Observe that $q^6/7>q^4$ and so $k_7(\varepsilon q)>5q^6/(6\cdot 7)>5q^4/6$.  Hence $5q^4/6d<k_j(u)\leq4u^2/3$, where $d=(5,q-\varepsilon)$. By exponents, we have $q^{40}<(8d/5)^{10}u^{20}<(8d/5)^{10}\cdot 5\cdot 43q^{29}$, whence $q\leq 9$. If $d=1$, then $q<3$ and so $q=9$. Solving the system $q^4<8u^2$,
$u^{20}<5\cdot 43\cdot q^{29}$, we deduce that $u=29,31$, and then $17\cdot193=k_8(q)\not\in \omega(S)$.

Suppose, finally, that $S\in\{S_{12}(u), O_{13}(u), O^-_{12}(u)\}$. Then one of $k_n(\varepsilon q)$  and $k_{n-1}(\varepsilon q)$ divides $k_{12}(u)$, while the other divides $k_{5}(\tau u)$ for some $\tau\in\{+,-\}$ or $k_8(u)$. 

If $n=10$, then $k_{10}(q)k_5(q)$ divides $k_j(u)$ for $j\in\{12,10,5,8\}$ and so $q^8/d<4u^4/3$, where $d=(5,q^2-1)$. This leads to a contradiction as in the case $S=S_{10}(u)$. 

If $n=9$, then $5q^6/6d<4u^4/3$, where $d=(3,q-\varepsilon)$. So $q^{36}<(8d/5)^6u^{24}<(8d/5)^{6}\cdot 4\cdot 43q^{29}$, whence $q\leq 7$.  If $d=1$, then $q=3$, and therefore $q=5,7$. Then $u\leq 13$ and we get a~contradiction verifying that $k_8(q)\not\in\omega(S)$.
\end{proof}

\begin{lemma}  $S\neq L_m^\tau(u)$. 
\end{lemma}

\begin{proof} Assume the opposite. Observe that every $r\in\pi(S)$ not adjacent to $2$ in $GK(S)$ lies in $R_m(\tau u)$ or $R_{m-1}(\tau u)$ by Lemma \ref{l:adj_criteria1}.

Let  $S=L_{10}^\tau(u)$.
Suppose that  $k_9(\varepsilon q)$ divides $k_{10}(\tau u)$. Then $5q^6/18<4u^4/3$. Since $2u^{32}<6\cdot 64q^{33}$ by Lemma \ref{l:exp}, it follows that $q^{48}<(8d/5)^8u^{32}<(8d/5)^8\cdot 6\cdot 32q^{33}$, whence $q=3$, a contradiction.  

Suppose that $k_9(\varepsilon q)$ divides $k_9(\tau u)$. By \cite[Lemma 2.5]{17Sta}, we have $3q^6<u^6$. If $n=10$, then $(u^9-\tau)/(10,u-\tau)$ divides $(q^9-\varepsilon)/(10,q-\varepsilon)$ by Lemma \ref{l:disjoint}(ii), which yields $u^9\leq 5q^9+6$. So $27q^{18}<u^{18}<(5q^9+6)^2$, whence $q<3$. Similarly, if $n=9$, then $(u^9-\tau)/(10,u-\tau)$ divides $(q^9-\varepsilon)/((9,q-\varepsilon)(q-\varepsilon))$ and so $u^9<20q^8$. Thus $3^3q^{18}<u^{18}<20^2q^{16}$, whence $q=3$, a contradiction. 

Since $k_9(\varepsilon q)$ divides neither $k_9(\tau u)$ nor $k_{10}(\tau u)$, by Lemma \ref{l:unique}, we conclude that $u-\tau$ divides $10$ and either $k_8(q)$ or $k_{10}(\varepsilon q)$ divides $k_9(\tau u)$. So $\tau u=+11,-9,-4$ and taking prime divisors of $k_9(\tau u)$ that are congruent 1 modulo $8$ or $10$, we are left with the equation $k_8(q)=9^6-9^3+1$, which has no solutions. 

Let $S=L_{12}^\tau(u)$. Suppose that either $k_n(\varepsilon q)$ or $k_{n-1}(\varepsilon q)$ divides $k_{12}(u)$. Using that $k_8(q)>q^4/2$ and $k_{10}(\varepsilon q)>5q^4/6(5,q+\varepsilon)$, we infer that $q^4<du^4$, where $d=\max\{2, 6(5,q+\varepsilon)/5\}$. Since $2u^{46}<15\cdot 64q^{33}$, we have $q^{46}<d^{23/2}\cdot 15\cdot 32q^{33}$, whence $q\leq 7$. Then $d=2$ and $q<3$.

By Lemma \ref{l:unique}, it follows that $u-\tau$ divides $12$ and one of $k_n(\varepsilon q)$ and $k_{n-1}(\varepsilon q)$ divides $k_{11}(\tau u)$. Arguing as in the case $m=10$, we 
derive a contradiction.

Let $S=L_{11}^\tau(u)$.  Suppose that there is $i\in\{7,9\}$ and $r\in R_i(\varepsilon q)$ such that
$t(r,S)<6$. Then by Lemma \ref{l:ts=tl+1}, we have that $r\in R_5(\tau u)$ and $r_8=r_8(\varepsilon q)$, $r_5=r_5(\varepsilon q)$ are large with respect to $S$. Note that $5r, vr\in\omega(S)$ and hence $v, 5\not\in R_4(q)$. By Lemmas \ref{l:r in K} and \ref{l:G/S-linear}, it follows that $r_4=r_4(q)$ is coprime to $|K|\cdot |\overline G/S|$, so $r_4\in\pi(S)$, $rr_4\not\in\pi(S)$ and $t(r_4,S)\geq 4$. This implies that $r_4$ is large with respect to $S$. Since $r_4$ is adjacent to both $r_5$ and $r_8$ in $GK(S)$ and $J(S)=E(S)$,  we have $e(r_8,u)=e(r_4,u)=e(r_5,u)$ Lemma \ref{l:tL}, whence $r_5r_8\in\omega(S)\setminus\omega(G)$. 

Thus there is $i\in\{7,9\}$ such that $k_i(\varepsilon q)$ divides $k_j(\tau u)$ for some $6\leq j\leq 10$, and therefore $5q^6/6d<4u^6/3$, where $d=\max\{(3,q-\varepsilon), (7,q-\varepsilon)\}$. If $n=9$, then $\frac{2}{118}u^{42}<86q^{29}$, and so 
$q^{42}<(8d/5)^7\cdot 118\cdot 43 q^{29}$. This yields $q\leq 7$, whence $d\leq 3$ and then $q=3$. Similarly, if $n=10$, then $q^{42}<(8d/5)^7\cdot 118\cdot 32 q^{33}$ and, therefore, $q\leq 13$. If $\varepsilon q\neq -13$, then $d\leq 3$ and so $q\leq 7$. If $d=1$, then $q=3$.
Hence $\varepsilon q=-13,+7,-5$, $i=7,9,9$, respectively, and $u\leq 8$. Now it is a routine to check that $k_i(\varepsilon q)$ does not divide $k_j(\tau u)$.

Let, finally, $S=L_{9}^\tau(u)$. If $n=10$, then by Lemma \ref{l:ts=tl}, at least one of $k_5(q)k_{10}(q)$, $k_9(\varepsilon q)$, $k_7(\varepsilon q)$ divides $k_j(\tau u)$ for some $j\in \{8,6,5\}$. Hence $5q^6/6d<\frac{u}{u-1}u^4$, where $d=\max\{(7,q-\varepsilon), (3,q-\varepsilon)\}$. This yields $u\geq 7$ and so $q^6<(7d/5)u^4$. It follows that 
$q^{42}<(7d/5)^7u^{28}\leq (7d/5)^7\cdot 16\cdot 64\cdot q^{33}$, whence $q\leq 11$. Then $d\leq 3$ and $q=5$, $u=7,8$. Now we see that $313=k_8(5)\not\in\omega(S)$. 

Let $n=9$. Suppose that one of $k_9(\varepsilon q)$ and $k_7(\varepsilon q)$ divides $k_j(\tau u)$ for $j\in\{8,6,5\}$. Then, as above, $u\geq 7$ and $q^6<(7d/5)u^4$. Thus $q^{42}<(7d/5)^7u^{28}\leq (7d/5)^7\cdot 16\cdot 86\cdot q^{29}$, and we conclude that $q=5$. Then $d\leq 3$ and $q=3$, a contradiction.

Since $r_9(\varepsilon q)$ is  not adjacent to $2$ in $GK(L)$, it follows that $k_9(\varepsilon q)$ divides $k_9(\tau u)$ and then $k_7(\varepsilon q)$ divides $k_7(\tau u)$. Hence $k_8(q)$ divides $k_8(u)$. It is easy to see that $k_8(q)\neq k_8(u)$ and as prime divisors of $k_8(u)$ are at least $17$, we have $8q^4<u^4$. On the other hand, $(u^9-\tau)/((9,u-\tau)(u-\tau))$ divides $(q^9-\varepsilon)/((9,q-\varepsilon)(q-\varepsilon))$, which yields $u^8<36q^8$, contradicting the previous inequality. This completes the proof of the lemma and Theorem~\ref{t:l9l10}.
\end{proof}

Now we a give a formal proof of Theorem \ref{t:main} supplementing the historical summary in the introduction. Let $L=L_n^\varepsilon(q)$. If $n\leq 4$ or $q$ is even, then the claim follows from \cite[Theorem 2.1]{23Survey}. In the remaining cases, it suffices to show that $L$ is almost recognizable since the description of $G$ such $L\leq G\leq \Aut L$ and $\omega(G)=\omega(L)$ is given in \cite{17Gre}.
If $n\in\{5,7\}$ and $\varepsilon=+$, then see the main result in \cite{12GrLyt.t}. If $n=5$ and $\varepsilon=-$, or $n=6$, then see \cite[Corollary 2]{24GrPan.t}. If $n=7$ and $\varepsilon=-$, then see \cite{26Pan.t}. If $n=8$, then see \cite[Theorem 1]{26GreRod_arxiv}. If $n\geq 11$, then see \cite[Theorem 1]{26Sta.t}. 
Finally, if $n=9$ or $10$, then see Theorem \ref{t:l9l10}.

\section{Proof of Theorem \ref{t:main2}}\label{s:t3}

Besides Theorem \ref{t:main2}, in this section we prove the following assertion.
 
\begin{prop}\label{p:2} Let $L$ and $S$ be as in Lemma {\rm \ref{l:reduction}}. Then $0\leq t(S)-t(L)\leq 1$. 
\end{prop}

By Lemma \ref{l:rest on t(S)}, to prove Proposition \ref{p:2}, it remains to handle the case where $L$ is a symplectic or orthogonal group with $t(L)=5$ or $6$.  So to deal with Theorem \ref{t:main2} and Proposition \ref{p:2} together, we suppose that $L$ is a symplectic or orthogonal group over a field of order $q$ and characteristic $p\neq 2$ such that $5\leq t(L)\leq 6$. 

Let $G$ be a finite group such that $\omega(G)=\omega(L)$ and $G$ is not an almost simple group with socle isomorphic to $L$. By Lemma \ref{l:reduction}, we have $S\leq \overline G=G/K\leq \Aut S$, where $S$ is a simple classical group over a field of order $u$ and characteristic $v\neq p$.

Observe that by Lemma \ref{l:t(S)>=t(L)}, if $i\in J(L)$, then $R_i(q)\cap(\pi(K)\cup\pi(\overline G/S))=\varnothing$, $k_i(q)\in\omega(S)$ and  $t(r_i(q),S)\geq t(L)$. 

\begin{lemma}\label{l:68} 
If $t(L)=6$, then $t(S)\neq 8$. 
\end{lemma}

\begin{proof}
Assume the opposite. Then $u^6<q^2$ by Lemma~\ref{l:d-estim}.
By Tables~\ref{tab:5tL8} and \ref{tab:tL}, at least one of $7$ and $14$ lies in $J(S)$, so $k_7(\epsilon q)\in\omega(S)$ for some suitable $\epsilon$. Also by Table~\ref{tab:5tL8}, we see that the Lie rank of $S$ is at most $15$. 
Applying Lemma \ref{l:orders}, we conclude that $3q^6/4d<2u^{15}$, where $d=(7,q-\epsilon)$. Together with  $u^6<q^2$, this yields $u^3<8d/3$, whence $u=2$.  Then $q\leq 9$ and $d=1$, a contradiction.
\end{proof}

In what follows, $t(L)=5$, that is, $L\in\{S_{10}(q), O_{11}(q)\}$ or $L\in\{S_{12}(q), O_{13}(q), O_{12}^-(q)\}$.

\begin{lemma}\label{l:57} 
If $t(L)=5$, then $t(S)\neq 7$. 
\end{lemma}
\begin{proof}
Assume the opposite. By Lemma~\ref{l:d-estim}, we have 
\begin{equation}\label{e:dif2}
u^3<q.    
\end{equation}
In particular, $q\geq 9$. By Table~\ref{tab:5tL8}, we see that either $\prk(S)\leq 10$, or $S=L_m^\tau(u)$ with $m=13,14$. 

Noting that $5,10\in J(L)$ and choosing $\epsilon\in\{+,-\}$  such that $(q-\epsilon,5)=1$,
we infer that $\Phi_5(\epsilon q)\in \omega(S)$. If  $\prk(S)\leq 10$, then  $9q^4/10<2u^{10}$ by Lemma \ref{l:orders}. Together with \eqref{e:dif2}, this yields $u^{2}<20/9$, a contradiction. If $S=L_m^\tau(u)$, 
where $u=2,3$, then $q=9,11$ if $u=2$ and $29\leq q\leq 43$ if $u=3$ and we verify that $r_{13}(\tau u)\not\in \pi(L)$. 

Let $S=L_{13}^\tau(u)$, where $u>3$. Applying Lemma \ref{l:ts=tl+1}, we conclude that there are $i\in\{5,10,8\}$ and $7\leq j\leq 12$ such that $k_i(q)$ divides $k_j(\tau u)$. Then $q^4/10<4u^{10}/3$. Together with \eqref{e:dif2}, this yields $u^{2}<40/3$ and so $u<4$, a contradiction. 

Let $S=L_{14}^\tau(u)$, where $u>3$. If $L\in\{S_{12}(q), O_{13}(q), O_{12}^-(q)\}$, then 
$u-\tau$ divides $14$ and $k_{12}(q)$ divides $k_{13}(\tau u)$ by Lemma \ref{l:s2n_even}. Let $L\in\{S_{10}(q), O_{11}(q)\}$ and suppose that there is  $i\in\{5,10\}$ such $k_i(q)$ divides $k_{14}(\tau u)$. Then $q^4/10<4u^6/3$, contradicting \eqref{e:dif2}. By Lemma \ref{l:s2n_odd},  we again have that $u-\tau$ divides $14$ and one of $k_5(q)$, $k_{10}(q)$ divides $k_{13}(\tau u)$. Thus $\tau u=-13$ or $+8$ and $k_j(q)$ divides $k_{13}(\tau u)$ for $j\in\{12,5,10\}$. Dropping prime divisors of $k_{13}(\tau u)$ not congruent 1 modulo $j$, we deduce that $q^4/10<k_j(q)\leq k_{13}(\tau u)/79<u^{12}/39$, contradicting \eqref{e:dif2}.
\end{proof}

Now Proposition \ref{p:2} follows from Lemmas \ref{l:rest on t(S)}, \ref{l:68} and \ref{l:57}.

\begin{lemma}\label{l:56} If $t(L)=5$, then $t(S)\neq 6$.
\end{lemma}

\begin{proof}
Assume the opposite. Then $S$ is one of $L^{\pm}_{11}(u)$, $L^{\pm}_{12}(u)$, $O_{15}(u)$, $S_{14}(u)$, $O^\pm_{14}(u)$, $O^+_{16}(u)$. 

We claim that $q\geq 5$ and  $u\geq 4$ and $S\neq O_{14}^+(5)$. Observe that $\Phi_5(\epsilon q)\in \omega(S)$ for any $\epsilon\in\{+,-\}$ such that $(q-\epsilon,5)=1$. Arguing as in the proof of Lemma \ref{l:u>3}, it is not hard to verify that $q\neq 3$. If $u\leq 3$, then $5q^4/6<\Phi_5(\epsilon q)< 3^{12}/2$, and so $q\leq 23$. If $S=O_{14}^+(5)$, then $7q^4/8<5^8/4$, whence $q\leq 17$. Now we verify that for every $u\in\{2,3,5\}$ and $q\leq 23$, there is $r\in R_5(q)$ such that $e(r,u)\geq 30$, contradicting the fact that $k_5(q)\in\omega(S)$.

Suppose that $L\in\{S_{10}(q), S_{11}(q)\}$.  Then $r_5(q)$ and $r_{10}(q)$ have disjoint neighborhoods in $GK(S)$ by Lemma \ref{l:disjoint}.

Let $S=S_{14}(u)$, $O_{15}(u)$, or $O_{16}^+(u)$. Applying Lemma \ref{l:unique}, we infer that $k_5(q)k_{10}(q)$ divides $k_7(u)k_{14}(u)$, whence $q^{10}<2du^{14}$, where $d=(5,q^2-1)$. On the other hand, $2u^{36}/5<\exp(S)\leq \exp(L)<20q^{21}$. It follows that $q\leq 7$, whence   $d=1$ and $q=3$, a contradiction. 

Let $S=O_{14}^\tau(u)$, where $\tau u\neq +5$, or $S=L_{12}^\tau(u)$. In the latter case, assume that $u-\tau$ does not divide $12$. By Lemma \ref{l:unique}, it follows that $k_5(\epsilon q)$ for some $\epsilon\in\{+,-\}$ divides $k_{12}(u)$ and so $5q^4/6d<u^4$, where $d=(5,q-\epsilon)$. By Lemma \ref{l:disjoint}(ii), we see that $(u^7-\tau)/(4,u-\tau)$ or $(u^{11}-\tau)/(12, u-\tau)$ divides $(q^5+\epsilon)/2$. In either case, $(u^7-1)/4\leq (q^5+1)/2$ and hence $u^7\leq 2q^5+3$. Thus $q^7<(6d/5)^{7/4}u^7<(6d/5)^{7/4}(2q^5+3)$, whence $q\leq 5$. Then $d=1$ and $q<3$. 

If $S=L_{12}^\tau(u)$, where $u-\tau$ divides $12$, and none of $k_5(q)$, $k_{10}(q)$ divides $k_{12}(u)$, then some of them, say $k_5(\epsilon q)$ divides $k_{11}(\tau u)$ by Lemma \ref{l:s2n_odd}. This implies that $(u^{11}-\tau)/(12, u-\tau)$ divides $(q^5-\epsilon)/2$, in particular, $u^{11}<6q^6+7$. Taking prime divisors of $k_{11}(\tau u)$ congruent 1 modulo 5, we conclude that $u=5$ or $13$ and $k_5(\epsilon q)\in\{5281, 18041, 12207031\}$. Since $u\geq 5$, it follows that $q\geq 23$, so  $k_5(\epsilon q)=12207031$. If $(5,q-\epsilon)=1$, then $q(q^2+1)(q-1)=2\cdot 3\cdot 5\cdot 11\cdot 71\cdot 521$, whence $q\in\{11,71,521\}$ and $(q^2+1_/2$ divides $5\cdot$, a contradiction. Similarly, if  $(5,q-\epsilon)=5$, then $q(q^2+1)(q-1)=2\cdot73\cdot 107\cdot3907$, so $(q^2+1)/2=73$, a contradiction.

Let $S=L_{11}^\tau(u)$. By Lemma \ref{l:ts=tl+1}, there is $i\in \{5,10,8\}$ and $6\leq j\leq 10$ such that $k_i(q)$ divides $k_j(u)$,  whence $q^4/2d<2u^6$, where $d=(5,q^2-1)$. Since $2u^{42}/118<\exp(S)\leq\exp(L)<20q^{21}$, it follows that $u^2<(1180)^{1/21}q$. Thus  $u^8<(1180)^{4/21}q^4<(1180)^{4/21}4du^6$, which yields $u\leq 8$. Solving the system $q^4<4du^6$, $u^2<(1180)^{1/21}q$ for $u=4,5,7,8$, we derive that $u=5$, $q=19$ and $i=10$. This is a contradiction since $k_{10}(19)=11\cdot 2251$  and $e(2251,5)=1125$.

Now suppose that $L\in\{S_{12}(q)$, $O_{13}(q)$, $O_{12}^-(q)\}$. Lemma \ref{l:s2n_even}
implies that $S$ is linear or unitary. If $S=L^\tau_{12}(u)$, then $u-\tau$ divides $12$ and $k_{12}(q)$ divides $k_{11}(\tau u)$, we derive a contradiction as above.

Let  $S=L^\tau_{11}(u)$. Then $k_{12}(q)$ divides $k_{11}(\tau u)$. If $R_i(q)\subseteq R_{j}(\tau u)$ for $i\in \{5,10,8\}$ and $j\in\{10,8,6\}$, then $q^4/6<4u^4/3$. If not, then by Lemma \ref{l:ts=tl+1}, it follows that every $r\in R_3(q)\cup R_6(q)$ is large with respect to $S$ and $k_3(q)k_6(q)$ divides  $k_{j}(\tau u)$ for $j\in\{10,8,6\}$, whence $q^4/3<4u^4/3$. In either case, $q^4<8u^4$. Since $2u^{42}/118<16q^{25}$, we have $q^{42}<(8)^{21/2}\cdot 118\cdot 8q^{25}$. This yields $q=5$ and then $u<4$. 
\end{proof}

\begin{lemma}\label{l:55}
If $t(L)=5$, then $t(S)\neq 5$.
\end{lemma}

\begin{proof}

Assume the opposite. Then $S$ is one of $L^{\pm}_{9}(u)$, $L^{\pm}_{10}(u)$, $L_{11}(2)$,  $S_{10}(u)$, $O_{11}(u)$, $S_{12}(u)$, $O_{13}(u)$, $O^-_{12}(u)$. 

Using that at least one of $k_5(q)$ and $k_{10}(q)$ lies in $\omega(S)$ and arguing as in the proof of Lemma \ref{l:56}, it is not hard to verify that $q\geq 5$ and $u\geq 4$. 

Suppose that $L\in\{O_{13}(q), S_{12}(q), O_{12}^-(q)\}$. Then $S\neq S_{10}(u),O_{11}(u)$ by Lemma \ref{l:s2n_even}.
By the same lemma, if $S\in \{S_{12}(u), O_{13}(u), O_{12}^-(u)\}$, then $k_{12}(q)$ divides $k_{12}(u)$, and so $(u^6+1)/(2,u-1)$ divides $(q^6+1)/2$.  The first divisibility implies that $q\leq u$, while the second implies that $u\leq q$, a contradiction. If $S=L_{10}^\tau(u)$, then $u-\tau$ divides $10$ and $k_{12}(q)$ divides $k_9(\tau u)$. Then $\tau u=-4,-9,+11$ and $k_{12}(q)\in\{37\cdot 109, 530713, 1772893\}$, which is impossible.   

Let $S=L_9^\tau(u)$. Lemma \ref{l:s2n_even} implies that $k_{12}(q)$ divides $k_9(\tau u)$ and hence $(u^9-\tau)/((9,u-\tau)(u-\tau))$ divides $(q^6+1)/2$, whence $u^8<9q^6$. By Lemma \ref{l:ts=tl}, for every $i\in\{5,10,8\}$ there is $j\in\{8,7,6,5\}$ such that $k_i(q)$ divides $k_j(\tau u)$, and also there is $j\in\{8,7,6,5\}$ such that $k_3(q)k_6(q)$ divides $k_j(u)$. So $q^4/10<k_6(\tau u)<2u^2$. It follows that $q^8<20^2u^4<20^23q^3$, which yields $q=3$.

Suppose that $L\in\{S_{10}(q), O_{11}(q)\}$.  

If $S\in \{O_{11}(u), S_{10}(u)\}$, then by Lemma \ref{l:unique}, the number $k_5(q)$ divides one of $k_5(u)$, $k_{10}(u)$, while $k_{10}(q)$ divides the other, in particular, $k_5(q^2)$ divides $k_5(u^2)$. Then $(u^5\pm 1)/(2,u-1)$ divides $(q^5\pm 1)/2$, and so $u^{10}-1$ divides $q^{10}-1$, whence $u<q$. As $\Phi_5(x)$ increases on $(1,+\infty)$, we conclude that $\Phi_5(u^2)<\Phi_5(q^2)$, which yields $(5,u^2-1)<(5,q^2-1)$, and therefore  $(5,u^2-1)=1$. Choosing $\epsilon$ such that $(5,q-\epsilon)=1$, we see that $k_5(\epsilon q)=\Phi_5(\epsilon q)\geq \Phi_{10}(q)>\Phi_5(q-1)\geq \Phi_5(u)\geq k_5(\pm u)$, a contradiction.

Let $S\in \{O_{13}(u), S_{12}(u), O_{12}^-(u)\}$. Then $k_5(\epsilon q)$ 
divides $k_{12}(u)$ for some $\epsilon$  by Lemma \ref{l:unique}. This implies that $(u^6+1)/(2,u-1)$ divides $(q^5-\epsilon)/2$. Hence $5q^4/6d<u^4$, where $d=(5,q-\epsilon)$ and $u^6<q^5$. It follows that $q^6<(6d/5)^{3/2}q^5$, whence $q\leq 13$. If $d=1$, then $q<3$. So either $\epsilon q=-9$, $u=5$ or $\epsilon q=+11$, $u=7$. In these cases $k_5(\epsilon q)$ does not divide $k_{12}(u)$.

Let $S=L_{10}^\tau(u)$. Suppose that $k_5(\epsilon q)$ 
divides $k_{10}(\tau u)$ for some $\epsilon$. Then $(u^{10}-1)/((10,u-1)(u-1)$ divides $(q^5-\epsilon)/2$. Hence $5q^4/6d<2u^4$,$d=(5,q-\epsilon)$, and $u^{9}/10<q^5$. It follows that $q^9<(12d/5)^{9/4}10q^5$, whence $q\leq 7$. Then $d=1$ and $q<3$. Thus by Lemma \ref{l:s2n_odd}, it follows that $u-\tau$ divides $10$ and one of $k_5(q)$ or $k_{10}(q)$ divides $k_9(\tau u)$. Then $\tau u=-4,-9,+11$ and non of primes divisors of $k_9(\tau u)$ is congruent 1 modulo $5$.

Let $S=L_9^\tau(u)$. Then $k_5(\epsilon q)$ divides $k_{9}(\tau u)$ for some $\epsilon$. This implies that  $(u^{9}-\tau)/((9,u-\tau)(u-\tau)$ divides $(q^5-\epsilon)/2$ and so $u^{8}<(9,u-\tau)q^5$. By Lemma \ref{l:ts=tl}, each of $k_5(-\epsilon q)$ and $k_8(q)$ divides $k_j(\tau u)$ for some $j\in\{8,7,6,5\}$ and, therefore, at least one of these numbers divides $k_j(\tau u)$ with $\varphi(j)\leq 4$. Hence $5q^4/6d<4u^4/3$, where $d=(5,q^2-1)$. It follows that $q^8<(8d/5)^{2}(9,u-\tau)q^5$, whence $q\leq 7$. Then $d=1$ and $q<3$. This completes the proof.
\end{proof}

Theorem \ref{t:main2} follows from Lemmas \ref{l:reduction}, \ref{l:rest on t(S)}, \ref{l:57}, \ref{l:56} and \ref{l:55}.

\section*{Acknowledgments}

The authors thank A. V. Zavarnitsine for his help with Mersenne primes, see Remark \ref{r:1} in the introduction.

M. A. Grechkoseeva and A. V. Vasil’ev acknowledge the support of the Russian Science Foundation, Project No. 24-11-00127, https://rscf.ru/en/project/24-11-00127/.



\begin{thebibliography}{10}

\bibitem{86Bang}
A.~S. Bang, {Taltheoretiske Unders{\o}gelser}, \emph{Tidsskrift Math.} \textbf{4} (1886), 70--80, 130--137.

\bibitem{91BrShi}
R.~Brandl and W.~J. Shi, {Finite groups whose element orders are consecutive integers}, \emph{J. Algebra} \textbf{143} (1991), no.~2, 388--400.

\bibitem{94BrShi}
R.~Brandl and W.~J. Shi, {The characterization of $PSL(2,q)$ by its element orders}, \emph{J. Algebra} \textbf{163} (1994), no.~1, 109--114.

\bibitem{08But.t}
A.~A. Buturlakin, {Spectra of finite linear and unitary groups}, \emph{Algebra Logic} \textbf{47} (2008), no.~2, 91--99.

\bibitem{10But.t}
A.~A. Buturlakin, {Spectra of finite symplectic and orthogonal groups}, \emph{Siberian Adv. Math.} \textbf{21} (2011), no.~3, 176--210.

\bibitem{85Atlas}
J.~H. Conway, R.~T. Curtis, S.~P. Norton, R.~A. Parker, and R.~A. Wilson, {Atlas of finite groups}, Clarendon Press, Oxford, 1985.

\bibitem{13Gor.t}
I.~B. Gorshkov, {Recognizability of alternating groups by spectrum}, \emph{Algebra Logic} \textbf{52} (2013), no.~1, 41--45.

\bibitem{16Gr.t}
M.~A. Grechkoseeva, {On spectra of almost simple groups with symplectic or orthogonal socle}, \emph{Siberian Math. J.} \textbf{57} (2016), no.~4, 582--588.

\bibitem{17Gre}
M.~A. Grechkoseeva, {On orders of elements of finite almost simple groups with linear or unitary socle}, \emph{J. Group Theory} \textbf{20} (2017), no.~6, 1191--1222.

\bibitem{18Gr.t}
M.~A. Grechkoseeva, {On spectra of almost simple extensions of even-dimensional orthogonal groups}, \emph{Siberian Math. J.} \textbf{59} (2018), no.~4, 623--640.

\bibitem{25Gr.t}
M.~A. Grechkoseeva, Recognizability of the groups ${PS}p_8(7^m)$ by the set of element orders, \emph{Math. Motes} \textbf{117} (2025), no.~3-4, 538--546.

\bibitem{12GrLyt.t}
M.~A. Grechkoseeva and D.~V. Lytkin, {Almost recognizability by spectrum of finite simple linear groups of prime dimension}, \emph{Siberian Math. J.} \textbf{53} (2012), no.~4, 645--655.

\bibitem{23Survey}
M.~A. Grechkoseeva, V.~D. Mazurov, W.~Shi, A.~V. Vasil'ev, and N.~Yang, Finite groups isospectral to simple groups, \emph{Commun. Math. Stat.} \textbf{11} (2023), 169--194.

\bibitem{24GrPan.t}
M.~A. Grechkoseeva and V.~V. Panshin, {On recognition of low-dimensional linear and unitary groups by spectrum}, \emph{Siberian Math. J.} \textbf{65} (2024), no.~5, 1074--1095.

\bibitem{26GreRod_arxiv}
M.~A. Grechkoseeva and V.~M. Rodionov, On recognition of simple classical groups with prime graph independence number 4 by spectrum, 2026, arXiv:2604.02885 [math.GR].

\bibitem{13GrShi.t}
M.~A. Grechkoseeva and W.~J. Shi, {On finite groups isospectral to finite simple unitary groups over fields of characteristic 2}, \emph{Siberian Electron. Math. Reps.} \textbf{10} (2013), 31--37.

\bibitem{15VasGr1}
M.~A. Grechkoseeva and A.~V. Vasil'ev, {On the structure of finite groups isospectral to finite simple groups}, \emph{J. Group Theory} \textbf{18} (2015), no.~5, 741--759.

\bibitem{19GrVasZv}
M.~A. Grechkoseeva, A.~V. Vasil'ev, and M.~A. Zvezdina, {Recognition of symplectic and orthogonal groups of small dimensions by spectrum}, \emph{J. Algebra Appl.} \textbf{18} (2019), no.~12, 1950230 [33 pages].

\bibitem{20GrZv.t}
M.~A. Grechkoseeva and M.~A. Zvezdina, {On recognition of $L_4(q)$ and $U_4(q)$ by spectrum}, \emph{Siberian Math. J.} \textbf{61} (2020), no.~6, 1039--1065.

\bibitem{Kou}
E.~I. Khukhro and V.~D.~Mazurov (Eds.), {{Unsolved problems in group theory. The Kourovka notebook}}, 2026, arXiv:1401.0300[math.GR].

\bibitem{94Maz.t}
V.~D. Mazurov, On the set of orders of elements of a finite group, \emph{Algebra Logic} \textbf{33} (1994), no.~1, 49--55.

\bibitem{98MazShi}
V.~D. Mazurov and W.~J. Shi, {A note to the characterization of sporadic simple groups}, \emph{Algebra Colloq.} \textbf{5} (1998), no.~3, 285--288.

\bibitem{12MazShi.t}
V.~D. Mazurov and W.~J. Shi, A criterion of unrecognizability by spectrum for finite groups, \emph{Algebra Logic} \textbf{51} (2012), no.~2, 160--162.

\bibitem{26Pan.t}
V.~Panshin, On recognition of simple groups with disconnected prime graphs by spectrum, \emph{Math. Notes}, to appear; see also arXiv:2509.03483[math.GR].

\bibitem{94PrShi}
C.~E. Praeger and W.~J. Shi, {A characterization of some alternating and symmetric groups}, \emph{Comm. Algebra} \textbf{22} (1994), no.~5, 1507--1530.

\bibitem{97Roi}
M.~Roitman, {On Zsigmondy primes}, \emph{Proc. Amer. Math. Soc.} \textbf{125} (1997), no.~7, 1913--1919.

\bibitem{84Shi}
W.~J. Shi, {A characteristic property of ${PSL}_2(7)$}, \emph{J. Aust. Math. Soc. Ser. A} \textbf{36} (1984), 354--356.

\bibitem{86Shi}
W.~J. Shi, {A characteristic property of $A_5$}, \emph{J. Southwest-China Teach. Univ.} \textbf{11} (1986), 11--14.

\bibitem{17Sta}
A.~Staroletov, {On almost recognizability by spectrum of simple classical groups}, \emph{Int. J. Group Theory} \textbf{6} (2017), no.~4, 7--33.

\bibitem{21Sta.t}
A.~M. Staroletov, {Composition factors of the finite groups isospectral to simple classical groups}, \emph{Siberian Math. J.} \textbf{62} (2021), no.~2, 341--356.

\bibitem{26Sta.t}
A.~M. Staroletov, {On recognition of linear and unitary groups by spectrum}, \emph{Siberian Math. J.} \textbf{67} (2026), no.~3.

\bibitem{05Vas.t}
A.~V. Vasil'ev, {On connection between the structure of a finite group and the properties of its prime graph}, \emph{Siberian Math. J.} \textbf{46} (2005), no.~3, 396--404.

\bibitem{15Vas}
A.~V. Vasil'ev, {On finite groups isospectral to simple classical groups}, \emph{J. Algebra} \textbf{423} (2015), 318--374.

\bibitem{08VasGr.t}
A.~V. Vasil'ev and M.~A. Grechkoseeva, {Recognition by spectrum for finite simple linear groups of small dimensions over fields of characteristic $2$}, \emph{Algebra Logic} \textbf{47} (2008), no.~5, 314--320.

\bibitem{09VasGrMaz.t}
A.~V. Vasil'ev, M.~A. Grechkoseeva, and V.~D. Mazurov, {On finite groups isospectral to simple symplectic and orthogonal groups}, \emph{Siberian Math. J.} \textbf{50} (2009), no.~6, 965--981.

\bibitem{14VasSt.t}
A.~V. Vasil'ev and A.~M. Staroletov, {Almost recognizability of simple exceptional groups of Lie type}, \emph{Algebra Logic} \textbf{53} (2015), no.~6, 433--449.

\bibitem{05VasVd.t}
A.~V. Vasil'ev and E.~P. Vdovin, {An adjacency criterion for the prime graph of a finite simple group}, \emph{Algebra Logic} \textbf{44} (2005), no.~6, 381--406.

\bibitem{11VasVd.t}
A.~V. Vasil'ev and E.~P. Vdovin, {Cocliques of maximal size in the prime graph of a finite simple group}, \emph{Algebra Logic} \textbf{50} (2011), no.~4, 291--322.

\bibitem{20YanGrVas}
N.~Yang, M.~A. Grechkoseeva, and A.~V. Vasil'ev, {On the nilpotency of the solvable radical of a finite group isospectral to a simple group}, \emph{J. Group Theory} \textbf{23} (2020), no.~3, 447--470.

\bibitem{04Zav}
A.~V. Zavarnitsine, {Recognition of the simple groups $L_3(q)$ by element orders}, \emph{J. Group Theory} \textbf{7} (2004), no.~1, 81--97.

\bibitem{04Zav1.t}
A.~V. Zavarnitsine, {The weights of irreducible ${SL}_3(q)$-modules in the defining characteristic}, \emph{Siberian Math. J.} \textbf{45} (2004), no.~2, 261--268.

\bibitem{06Zav.t}
A.~V. Zavarnitsine, {Recognition of the simple groups $U_3(q)$ by element orders}, \emph{Algebra Logic} \textbf{45} (2006), no.~2, 106--116.

\bibitem{Zs}
K.~Zsigmondy, {Zur Theorie der Potenzreste}, \emph{Monatsh. Math. Phys.} \textbf{3} (1892), 265--284.

\end{thebibliography}

\end{document}